\documentclass[final,11pt]{elsarticle}
\usepackage{srcltx}
\usepackage{eurosym}
\usepackage{mathtools}
\usepackage{amsmath}
\usepackage{amsfonts}
\usepackage{amssymb}
\usepackage{amsthm}
\usepackage{graphicx}
\usepackage{mathrsfs}
\usepackage{xcolor}
\usepackage{exscale}
\usepackage{latexsym}
\usepackage[pagewise]{lineno}
\usepackage[colorlinks,plainpages=true,pdfpagelabels,hypertexnames=true,colorlinks=true,pdfstartview=FitV,linkcolor=blue,citecolor=red,urlcolor=black]{hyperref}
\PassOptionsToPackage{unicode}{hyperref}
\PassOptionsToPackage{naturalnames}{hyperref}
\usepackage{enumerate}
\usepackage[shortlabels]{enumitem}
\usepackage{bookmark}
\usepackage{wasysym}
\usepackage{esint}
\usepackage[ddmmyyyy]{datetime}
\usepackage[margin=2cm]{geometry}
\makeatletter
\g@addto@macro\normalsize{%
	\setlength\abovedisplayskip{4pt}
	\setlength\belowdisplayskip{4pt}
	\setlength\abovedisplayshortskip{4pt}
	\setlength\belowdisplayshortskip{4pt}
}
\numberwithin{equation}{section}
\everymath{\displaystyle}
\usepackage[capitalize,nameinlink]{cleveref}
\crefname{section}{Section}{Sections}
\crefname{subsection}{Subsection}{Subsections}
\crefname{condition}{Condition}{Conditions}
\crefname{hypothesis}{Hypothesis}{Conditions}
\crefname{assumption}{Assumption}{Assumptions}
\crefname{lemma}{Lemma}{Lemmas}
\crefname{definition}{Definition}{Definitions}

\crefformat{equation}{\textup{#2(#1)#3}}
\crefrangeformat{equation}{\textup{#3(#1)#4--#5(#2)#6}}
\crefmultiformat{equation}{\textup{#2(#1)#3}}{ and \textup{#2(#1)#3}}
{, \textup{#2(#1)#3}}{, and \textup{#2(#1)#3}}
\crefrangemultiformat{equation}{\textup{#3(#1)#4--#5(#2)#6}}%
{ and \textup{#3(#1)#4--#5(#2)#6}}{, \textup{#3(#1)#4--#5(#2)#6}}%
{, and \textup{#3(#1)#4--#5(#2)#6}}

\Crefformat{equation}{#2Equation~\textup{(#1)}#3}
\Crefrangeformat{equation}{Equations~\textup{#3(#1)#4--#5(#2)#6}}
\Crefmultiformat{equation}{Equations~\textup{#2(#1)#3}}{ and \textup{#2(#1)#3}}
{, \textup{#2(#1)#3}}{, and \textup{#2(#1)#3}}
\Crefrangemultiformat{equation}{Equations~\textup{#3(#1)#4--#5(#2)#6}}%
{ and \textup{#3(#1)#4--#5(#2)#6}}{, \textup{#3(#1)#4--#5(#2)#6}}%
{, and \textup{#3(#1)#4--#5(#2)#6}}

\crefdefaultlabelformat{#2\textup{#1}#3}
\numberwithin{equation}{section}

\newtheorem{theorem} {Theorem}[section]
\newtheorem{proposition}[theorem]{Proposition}
\newtheorem{lemma}[theorem]{Lemma}

\newtheorem{example}[theorem]{Example}
\newtheorem{counter example}[theorem]{Counter Example}
\newtheorem{remark}[theorem] {Remark}
\newtheorem{definition}[theorem] {Definition}

\newtheorem{claim}[theorem] {Claim}

\def\N{\mathbb{N}}
\def\CC{{\rm \kern.24em \vrule width.02em height1.4ex depth-.05ex \kern-.26emC}}

\def\TagOnRight

\def\AA{{it I} \hskip-3pt{\tt A}}

\def\QQ{\rlap {\raise 0.4ex \hbox{$\scriptscriptstyle |$}} {\hskip -0.1em Q}}

\makeatletter
\newcommand{\vo}{\vec{o}\@ifnextchar{^}{\,}{}}
\makeatother

\def\YYint#1#2#3{{\setbox0=\hbox{$#1{#2#3}{\iint}$}
		\vcenter{\hbox{$#2#3$}}\kern-.50\wd0}}

\def\XXint#1#2#3{{\setbox0=\hbox{$#1{#2#3}{\int}$}
		\vcenter{\hbox{$#2#3$}}\kern-.50\wd0}}

\makeatletter
\def\namedlabel#1#2{\begingroup
	\def\@currentlabel{#2}%
	\label{#1}\endgroup
}
\makeatother
\makeatletter
\newcommand{\rmh}[1]{\mathpalette{\raisem@th{#1}}}
\newcommand{\raisem@th}[3]{\hspace*{-1pt}\raisebox{#1}{$#2#3$}}
\makeatother

\newcounter{desccount}
\newcommand{\descitem}[2]{\item[#1]\refstepcounter{desccount}\label{#2}}
\newcommand{\descref}[2]{\hyperref[#1]{\textnormal{\textcolor{black}{}\textcolor{blue}{ #2}\textcolor{black}{}}}}

\newcommand{\dref}[2]{\hyperref[#1]{\textcolor{black}{(}\textcolor{blue}{\bf #2}\textcolor{black}{)}}}
\newcommand{\be} {\begin{eqnarray}}
	\newcommand{\ee} {\end{eqnarray}}
\newcommand{\Bea} {\begin{eqnarray*}}
	\newcommand{\Eea} {\end{eqnarray*}}
\newcommand{\pa} {\partial}

\newcommand{\al} {\alpha}
\newcommand{\rr}{\rightarrow}

\newcommand{\B} {\beta}
\newcommand{\de} {\delta}

\newcommand{\p}  {\prime}
\newcommand{\e}  {\epsilon}

\newcommand{\la} {\lambda}
\newcommand{\si} {\sigma}

\newcommand{\La} {\Lambda}

\newcommand{\f}{\infty}
\newcommand{\R}{\mathbb{R}}

\newcommand{\noi} {\noindent}

\newcommand{\ga}{\gamma}

\DeclareMathOperator{\dv}{div}

\newcommand{\norm}[1]{\left|\hspace{-0.2mm}\left| #1 \right|\hspace{-0.2mm}\right|}
\newcommand{\abs}[1]{\left| #1\right|}

\newcounter{whitney}
\refstepcounter{whitney}

\newcounter{ineqcounter}
\refstepcounter{ineqcounter}
\makeatletter
\def\ps@pprintTitle{%
	\let\@oddhead\@empty
	\let\@evenhead\@empty
	\def\@oddfoot{}%
	\let\@evenfoot\@oddfoot}
\makeatother

\newcommand{\refcheckize}[1]{%
	\expandafter\let\csname @@\string#1\endcsname#1%
	\expandafter\DeclareRobustCommand\csname relax\string#1\endcsname[1]{%
		\csname @@\string#1\endcsname{##1}\wrtusdrf{##1}}%
	\expandafter\let\expandafter#1\csname relax\string#1\endcsname
}
\makeatother

\refcheckize{\cref}
\refcheckize{\Cref}


\makeatletter
\newcommand{\mainsectionstyle}{%
	\renewcommand{\@secnumfont}{\bfseries}
	\renewcommand\section{\@startsection{section}{2}%
		\z@{.5\linespacing\@plus.7\linespacing}{-.5em}%
		{\normalfont\bfseries}}%
}
\makeatother
\usepackage{pgf,tikz}
\usetikzlibrary{arrows}
\usetikzlibrary{decorations.pathreplacing}

\usepackage{xpatch}
\xpatchcmd{\MaketitleBox}{\hrule}{}{}{}
\xpatchcmd{\MaketitleBox}{\hrule}{}{}{}

\date{}

\let\oldbibliography\thebibliography
\renewcommand{\thebibliography}[1]{\oldbibliography{#1}
	\setlength{\itemsep}{0pt}}

\allowdisplaybreaks

\usepackage{lineno,hyperref}

\newcommand{\U}{\textbf{u}}
\newcommand{\V}{\textbf{v}}
\newcommand{\W}{\textbf{w}}

\begin{document}
	\begin{frontmatter}
		
        \title{On blow up of $C^1$ solutions of isentropic Euler system}
		\author[myaddress-1]{Shyam Sundar Ghoshal}\ead{ghoshal@tifrbng.res.in}
		\author[myaddress-2]{Animesh Jana}\ead{animesh.jana@dauphine.psl.eu}
		\address[myaddress-1]{Centre for Applicable Mathematics, Tata Institute of Fundamental Research, Post Bag No 6503, Sharadanagar, Bangalore - 560065, India.}
		\address[myaddress-2]{CEREMADE (CNRS UMR no. 7534), Universit\'e Paris Dauphine, PSL Research University, Place du Mar\' echal De Lattre De Tassigny, Paris CEDEX 16, 75775 France.}
		
		\begin{abstract}
		  In this article, we study the break-down of smooth and continuous solutions to isentropic Euler system in multi dimension. Sideris [Comm. Math. Phys. 1985]  proved the blow up of smooth solutions when initial data satisfies an `integral condition'. We show that a $C^1$ solution of isentropic Euler equation breaks down if (i) gradient of initial velocity has a negative real eigenvalue at some point $x_0\in\R^d$ and (ii) Hessian of initial density satisfies a smallness condition in Sobolev space. Our proof also works for the data which fails to satisfy the above-mentioned `integral condition'. Furthermore, we prove the global existence of smooth solution when (i) eigenvalues of gradient of initial velocity have non-negative real-part and (ii) initial density satisfies a smallness condition. This extends the global existence result of [Grassin, Indiana Univ. Math. J. 1998]. Another goal of this article is to study the breakdown of continuous weak solutions of isentropic Euler equations. We are able to show that the `integral condition' of Sideris can cause the breakdown of continuous solutions in finite time. This improves the blow up result of Sideris from $C^1$ to continuous space. 
		  
		\end{abstract}

%
%


\end{frontmatter}	
	\tableofcontents

\section{Introduction}
In this article, we are interested in regularity and breakdown of smooth solution to the isentropic Euler equations which reads as follows,
\begin{align}
	\pa_t\varrho+\dv_x(\varrho\U)&=0\mbox{ for }x\in \R^d,t>0,\label{eqn-isen-1}\\
	\pa_t(\varrho\U)+\dv_x(\varrho\U\otimes\U)+\nabla_xp(\varrho)&=0\mbox{ for }x\in \R^d,t>0,\label{eqn-isen-2}
\end{align}
where $\varrho,\U$ denote the density and velocity of a compressible fluid and the function $p:[0,\f)\rr[0,\f)$ represents the pressure. Here we consider $p(\varrho)=\varrho^\gamma$ for $\gamma>1$.

Despite of the importance of PDE \eqref{eqn-isen-1}--\eqref{eqn-isen-2} in the physical world, very little is known about well-posedness. For sufficiently smooth initial data, the local-in-time existence of smooth solutions is known due to \cite{Chemin,Kato,KMU}. In one dimension, the $C^1$--norm of solution blows up in finite time for compact support velocity (see \cite{John,Lax-jmp,Liu,KMU}). In his seminal work \cite{Sideris-cmp}, Sideris proposed a sufficient condition for multi-dimension to show that the smooth solution breaks down in finite time. See also \cite{BuckShkVic,Christodoulou,Christodoulou-ems,JWX,Lai-Schi,Majda-existence,Sideris-hyp} for results on blow up and lifespan of smooth solution to isentropic Euler system and related hyperbolic system of conservation laws. In this article, we first prove that for a certain restriction on initial data, the continuity of a solution breaks down in finite time if it also satisfies the entropy inequality in the sense of distribution,
\begin{equation}\label{ineq:entropy}
	\pa_t\left[\frac{1}{2}\varrho\abs{\U}^2+P(\varrho)\right]+\dv_x\left[\left(\frac{1}{2}\varrho\abs{\U}^2+P(\varrho)+p(\varrho)\right)\U\right]\leq0,\mbox{ where }P(\varrho)=\varrho\int\limits_{0}^{\varrho}\frac{p(r)}{r^2}dr.
\end{equation}
 We note that the initial data for which the result of Sideris \cite{Sideris-cmp} is valid also satisfies the assumption of our result. Hence, for the initial data considered in \cite{Sideris-cmp}, the corresponding solution not only loses its $C^1$ regularity in finite time, it also becomes discontinuous.
 
 Recently, there has been a growing interest in existence of continuous solutions for fluid flow equations (see \cite{DK,GK,Isett}). For isentropic Euler equation, it has been shown \cite{GK} that there are infinitely many H\"older continuous weak solutions from a specific initial data. It is then natural to ask whether these solutions are globally defined or not. We show that if a continuous solution satisfies entropy inequality \eqref{ineq:entropy} and an `integral condition' (as in \eqref{condition-Sideris} below), then their continuity breaks down in finite time.

On the other hand, global smooth solution exists \cite{Lax-jmp} in 1-D for a class of initial data. Chemin \cite{Chemin} proved the global existence of smooth solution in multi-D when initial density has finite second spatial moment. Later, Serre \cite{Serre} showed that global-in-time solution exists when the initial density and velocity, $(\varrho_0,\U_0)$ are sufficiently close to $(0,Ax)$ in an appropriate Sobolev norm and $A$ is a positive definite constant matrix. The theory of global existence for isentropic Euler system has been extended by Grassin \cite{Grassin} for a larger class of initial data. In the article \cite{Grassin}, it has been shown that smooth solution exists for all time $t>0$ if (i) the initial density is sufficiently small in Sobolev norm and (ii) eigenvalues of $\nabla_x\textbf{u}_0$ belong to a compact set of $\mathbb{C}\setminus(-\f,0)$. 
In this article, we establish a relation between the eigenvalues of $\nabla_x\textbf{u}_0$ and the global existence of smooth solution. We consider the background velocity $\overline{\U}$ solving the following Cauchy problem 
\begin{equation}\label{eqn-Burgers-intro}
	\left.\begin{array}{rl}
		\pa_t\overline{\U}+\overline{\U}\cdot\nabla_x\overline{\U}&=0\,\,\,\,\mbox{ for }x\in\R^d,t>0,\\
		\overline{\U}(0,x)&=\overline{\U}_0\mbox{ for }x\in\R^d.
	\end{array}\right\}
\end{equation}
For the vectorial Burgers equation \eqref{eqn-Burgers-intro}, global smooth solution exists if and only if the initial data $\overline{\U}_0$ eigenvalues of $\nabla_x\overline{\textbf{u}}_0$ belong to a compact set of $\mathbb{C}\setminus(-\f,0)$, 
 equivalently, we have finite time blow up of $C^1$ solutions of \eqref{eqn-Burgers-intro} if and only if $\nabla_x\overline{\textbf{u}}_0$ has at least one negative real eigenvalue. In this article, we want to show an analogous result for isentropic Euler system \eqref{eqn-isen-1}--\eqref{eqn-isen-2}. We first prove the finite time blow up of smooth solution when $\nabla_x{\textbf{u}}_0(x_0)$ has at least one negative real eigenvalue for some $x_0\in\R^d$. In the spirit of \cite{Grassin}, we also show that global smooth solution exists for $\inf\limits_{x\in\R^d}\min\limits_{\xi\in\mathbb{S}^{d-1}}\nabla_x\U_0(x):\xi\otimes\xi\geq0$ where $\mathbb{S}^{d-1}$ is the unit sphere in $\R^d$ and initial density satisfies a smallness condition. Combining these two results, we conclude the characterization of initial velocity providing global smooth solutions. 



 Though our main concern of this article is to find the  optimal characterization of initial data causing blow up of $C^1$ or continuous solution, it is also an interesting subject to study the description and stability of the shock and how the solution behaves after the blow up time. Majda \cite{Majda-existence,Majda-stability} studied the existence and stability of shock for multi-D system of conservation laws. Description of shock formation has been studied in \cite{Christodoulou,LS-Invent}. We also refer recent results \cite{BuckShkVic} in this direction. For spherically symmetric smooth initial data, the blow up of $L^\f$-norm of density and velocity has been shown in \cite{MRRS-II}.
 
 In a forthcoming paper \cite{GJ-NS}, we obtain analogous results for compressible Navier-Stokes equation with low regularity.
 
\section{Main results}

\subsection{Blow up of continuous solution}
Before we state our result on break-down of continuous solution we provide the definition of weak solution. 
\begin{definition}\label{defn:admissible}
	Let $T^*>0$. We say $(\varrho,\U)\in C([0,T^*);L^1_{loc}(\R^d))\cap L^\f([0,T^*)\times\R^d)$ is a weak solution to \eqref{eqn-isen-1}--\eqref{eqn-isen-2} if the following holds.
	\begin{itemize}
		\item For all $0\leq t_1<t_2<T^*$ and $\varphi\in C^1_c([t_1,t_2]\times\R^d)$, 
		\begin{equation}
			\left[\int\limits_{\R^d}\varrho(t,\cdot)\varphi(t,\cdot)\,dx\right]_{t=t_1}^{t=t_2}=\int\limits_{t_1}^{t_2}\int\limits_{\R^d}\left[\varrho\pa_t\varphi+\varrho\U\cdot\nabla_x\varphi\right]\,dxdt.
		\end{equation}
     	\item For all $0\leq t_1<t_2<T^*$ and $\psi\in C^1_c([t_1,t_2]\times\R^d;\R^d)$, 
	\begin{equation}
		\left[\int\limits_{\R^d}\varrho\U(t,\cdot)\cdot\psi(t,\cdot)\,dx\right]_{t=t_1}^{t=t_2}=\int\limits_{t_1}^{t_2}\int\limits_{\R^d}\left[\varrho\U\cdot\pa_t\psi+\varrho\U\otimes\U:\nabla_x\psi+p(\varrho)\dv_x\psi\right]\,dxdt.
	\end{equation}
	\end{itemize}
A weak solution $(\varrho,\U)$ will be called as admissible solution if it satisfies the following inequality
\begin{align}
	&\int\limits_{0}^{T^*}\int\limits_{\R^d}\left[\left(\frac{\varrho\abs{\U}^2}{2}+P(\varrho)\right)\pa_t\varphi+\left(\frac{\varrho\abs{\U}^2}{2}+P(\varrho)+p(\varrho)\right)\U\cdot\nabla_x\varphi\right]dxdt\nonumber\\
	&+ \int\limits_{\R^d}\left(\frac{\varrho_0\abs{\U_0}^2}{2}+P(\varrho_0)\right)\varphi(0,\cdot)\,dx\geq0,
\end{align}
for all $0\leq\varphi\in C^1_c([0,T^*)\times\R^d)$ where $P(\varrho)=\frac{1}{\gamma-1}\varrho^\gamma$.
\end{definition}
The well-posedness of admissible solution to isentropic Euler system is largely open. In one dimension the well-posedness for admissible weak solution is known in BV setting \cite{Bressan-Liu-Yang,Bressan-WFT,Glimm}. But, in multi-D even for Riemann type planar initial data it has been shown \cite{CDK} that infinitely many $L^\f$ admissible weak solutions exist. In three dimension, infinitely many continuous admissible weak solution has been constructed \cite{GK} so that they satisfy same initial data. On the other hand, uniqueness has been shown \cite{FGJ} when solution has sufficient H\"older regularity and the velocity gradient satisfies a one-sided bound condition. See also \cite{GJK,GJW,GJ-full-Euler} for recent results on uniqueness of H\"older continuous solutions for compressible Euler system and relation hyperbolic conservation laws.

To study the breakdown of continuity of weak solutions, we impose the following condition on initial data as in \cite{Sideris-cmp},
\begin{equation}\label{condition-Sideris}
	\frac{1}{\omega_dR^{d+1}}\int\limits_{\R^d}\varrho_0\U_0\cdot x\,dx\geq  (d+1) \si \norm{\varrho_0}_{L^\f(\R^d)},
\end{equation}
where $\si=\sqrt{\gamma}(\overline{\varrho})^{\frac{\gamma-1}{2}}$ and $\omega_d$ denotes the area of $d$ dimensional sphere. Furthermore, we additionally assume that $(\varrho_0,\U_0)$ satisfies 
\begin{equation}\label{assumption-support}
	(\varrho_0(x),\U_0(x))=(\overline{\varrho},\textbf{0})\mbox{ for }\abs{x}\geq R\mbox{ and }	\int\limits_{\R^d}(\varrho_0-\overline{\varrho})\,dx\geq0,
\end{equation}
for some $\overline{\varrho}>0$ and $R>0$. We show that a continuous admissible solution breaks down in finite time if the initial data satisfies above two conditions \eqref{condition-Sideris} and \eqref{assumption-support}. More precisely, we prove the following result.
\begin{theorem}\label{theorem-cont}
	Let $(\varrho,\U)\in C([0,T^*)\times \R^d)$ be an admissible solution to \eqref{eqn-isen-1}--\eqref{eqn-isen-2} with initial data $(\varrho_0,\U_0)$according to Definition \ref{defn:admissible}. We assume that the initial data $(\varrho_0,\U_0)$ satisfies \eqref{condition-Sideris} and \eqref{assumption-support}. Then, $T^*$ is finite.
\end{theorem}

Sideris \cite{Sideris-cmp} proved that $C^1$ solutions break down in finite time if the initial data satisfies the conditions \eqref{condition-Sideris} and \eqref{assumption-support}. With a similar condition, blow up of classical solution has been shown \cite{Xin} Navier-Stokes. We refer to \cite{Suzuki,Yuen}  and references therein for more results on break-down of $C^1$ solutions with a condition like \eqref{condition-Sideris}. 

We note that the entropy condition \eqref{ineq:entropy} and the continuity of solution plays key role to prove that characteristics generated from $\{\abs{x}\leq R\}$ are confined inside the cone $\{\abs{x}\leq R+\si t\}$. For $C^1$ solutions we do not additionally need to assume \eqref{ineq:entropy} as it can be shown by using \eqref{eqn-isen-1}--\eqref{eqn-isen-2} (see \cite{Dafermos}) that equality holds in \eqref{ineq:entropy} for $C^1$ solutions. For isentropic Euler equations, equality in \eqref{ineq:entropy} can still hold even for weak solutions which may not be $C^1$ (see \cite{BGSTW,FGGW,GMS}). Furthermore, we would like to mention a recent result \cite{GK} on existence of infinitely many continuous admissible solutions of isentropic Euler equations. For some specific choice of continuous initial data, this \cite{GK} gives an existential results for \eqref{eqn-isen-1}--\eqref{eqn-isen-2}. We wish to remark that due to our result in Theorem \ref{theorem-cont} the solutions constructed in \cite{GK} can not be extended globally (in time) if the initial data satisfies the integral condition \eqref{condition-Sideris}.

\subsection{Blow up of $C^1$ solution}
In the view of blow up result of Sideris \cite{Sideris-cmp}, it is natural to ask whether global $C^1$ solution exists if a $C^1$ initial data does not satisfy the condition \eqref{condition-Sideris}. We answer to this question in our next result. To prove a finite time blow up of $C^1$ solution, we impose the following conditions on the initial data. 
	\begin{description}
	\descitem{(ND)}{Cond-1}  {\em Negative definiteness of velocity gradient:} There exists $\la_{max}>0$ and $x_0\in\R^d$ such that  
	\begin{equation}\label{condition-neg-1}
		\nabla_x\U_0(x_0)\mbox{ is symmetric and }\nabla_x\U_0(x_0)\xi_0=-\la_{max}\xi_0\mbox{ for some }\xi_0\in\mathbb{S}^{d-1}.
	\end{equation}
%
%
\end{description}

We now state our result on blow up of $C^1$ solutions of isentropic Euler system \eqref{eqn-isen-1}--\eqref{eqn-isen-2}.
\begin{proposition}\label{theorem:blow-up-C1-1}
	Let $(\varrho,\U)\in C^1([0,T^*]\times\R^d)$ be a solution of isentropic Euler system \eqref{eqn-isen-1}--\eqref{eqn-isen-2} corresponding to the initial data $(\varrho_0,\U_0)$ satisfying \descref{Cond-1}{(ND)} and $\U_0(x_0+r\xi_0)-\U_0(x_0)=-\la_0r\xi_0$ for some $r>0,\la_0>0,\xi_0\in\mathbb{S}^{d-1}$ and $x_0\in\R^d$. Define $M:=\norm{\nabla_x\U}_{L^\f([0,T^*]\times\R^d)}$. 
	If $\varrho,\U$ and $\la_0$ satisfy the following estimate,
	\begin{equation}\label{condition-eps-1}
		\abs{\la_{max}-\la_0}+\norm{\pi(\cdot)\nabla_x\pi(\cdot)}_{L^\f([0,1/\la_{max}]\times\R^d)}\leq\e_0:=\frac{\la_0r}{2}\left[\frac{rM}{\la_{max}}+\frac{2e^{\frac{M}{\la_{max}}}}{M}\right]^{-1},
	\end{equation}
   where $\pi=\sqrt{\frac{\gamma-1}{4\gamma}}\varrho^{\frac{\ga-1}{2}}$, then we have $T^*\leq 1/\la_{max}$.
\end{proposition}
\begin{remark}\label{remark-prop-2.3}
	\noi\begin{enumerate}
		\item We would like to mention that in order to prove the blow up result, the smallness assumption on $\pi\nabla_x\pi$ as in \eqref{condition-eps-1} is not essentially required on whole $\R^d$. One can possibly replace it by a locally smallness assumption on the domain of dependence. 
		\item We note that the condition on $\U_0$ of having a direction such that $\U_0(x_0+r\xi_0)-\U_0(x_0)=-\la_0r\xi_0$ holds can be replaced by assuming that $\nabla_x\U_0(x)$ has $d$ real eigenvalues in a neighborhood of $x_0\in\R^d$. When $\nabla_x\U_0$ has $d$ real eigenvalues then it is guaranteed that there exists a direction such that  $\U_0(x_0+r\xi_0)-\U_0(x_0)=-\la_0r\xi_0$ holds. This can be shown in a similar manner as in Lax theory for solving Riemann problem of $n\times n$ hyperbolic system in 1-D (see \cite{Bressan-book,Lax-1957}).
	\end{enumerate}
\end{remark}
Since the smallness condition \eqref{condition-eps-1} in Proposition \ref{theorem:blow-up-C1-1} is given on the solution, it does not directly serve our purpose to classify the initial data causing finite time blow up of $C^1$. We want to replace \eqref{condition-eps-1} by considering smallness on $H^m$-norm of $\nabla^2_x\varrho^{\frac{\ga-1}{2}}_0$ and $\nabla^2_x\U_0$ and the blow up result holds true.
\begin{theorem}\label{theorem:blow-up-C1-2}
	Let $(\varrho,\U)\in C^1([0,T^*)\times\R^d)$ be a solution of isentropic Euler system \eqref{eqn-isen-1}--\eqref{eqn-isen-2} corresponding to the initial data $(\varrho_0,\U_0)$ satisfying 
	\begin{equation}\label{smallness-Hm}
	\left(\norm{\nabla^2_x\varrho^{\frac{\ga-1}{2}}_0}_{H^m(\R^d)}+\norm{\nabla^2_x\U_0}_{H^m(\R^d)}\right)\cdot\norm{\varrho^\frac{\ga-1}{2}_0}_{L^\f(\R^d)}\leq \e_0\mbox{ for some integer }m>1+\frac{d}{2},\e_0>0. 
	\end{equation}
	We assume that $\U_0$ satisfies  \descref{Cond-1}{(ND)} for some $\la_{max} >0$. Then, $T^*$ is finite provided $\e_0<\frac{\la^2_{max}}{5(\gamma-1)}$.
\end{theorem}

\begin{remark}
	\noi\begin{enumerate}
		\item We remark that there exists initial data $\varrho_0,\U_0$ such that it satisfies \descref{Cond-1}{(ND)} and \eqref{smallness-Hm} but it fails to satisfy \eqref{condition-Sideris}. See Example \ref{example-2} in section \ref{sec:thm-2}. Therefore, for the $C^1$ solutions arsing from initial data as in Example \ref{example-2}, we can not invoke blow up result of Sideris \cite{Sideris-cmp} but Theorem \ref{theorem:blow-up-C1-1} is still applicable as they satisfies \eqref{smallness-Hm}.
		\item Similar to (1) of Remark \ref{remark-prop-2.3}, we again mention that one can obtain the blow up result of Theorem \ref{theorem:blow-up-C1-2} by using a local smallness condition with a proper use of finite speed of propagation.
	\end{enumerate}
\end{remark}

We prove Theorem \ref{theorem:blow-up-C1-2} in a similar technique as Proposition \ref{theorem:blow-up-C1-1} with the help of an energy estimate of Kato \cite{Kato} for $C^1$ solutions.

\subsection{Global existence of smooth solution}
We remark that the result of Theorem \ref{theorem:blow-up-C1-2} complements the global existence theory developed in \cite{Chemin,Grassin,Serre}. We recall that if initial data satisfies (i) $\min\limits_{\xi\in\mathbb{S}^{d-1}}\nabla_x\U_0(x):\xi\otimes\xi\geq c>0$ for all $x\in\R^d$ and (ii) $\pi_0$ is sufficiently small then the global smooth solution exists (see \cite{Grassin}). As Theorem \ref{theorem:blow-up-C1-2} rules out the possibility of global solution when velocity gradient $\nabla_x\U_0$ has at least one negative real eigenvalue, now it is natural to ask whether global solution when $\nabla_x\U_0(x):\xi\otimes\xi\geq0$ for all $\xi\in\mathbb{S}^{d-1}$. We answer this question in our next result.
\begin{proposition}\label{theorem:global-existence}
	Let $m>1+\frac{d}{2}$ and $\varrho_0,\U_0$ be satisfying 
	\begin{description}
		\descitem{(G-1)}{hyp-1} $\nabla_x^2\U_0\in H^{m-1}(\R^d)$ and $\nabla_x\U_0\in L^\f(\R^d)$.
		\descitem{(G-2)}{hyp-2} $\inf\left\{\nabla_x\U_0(x):\xi\otimes\xi;\,\xi\in\mathbb{S}^{d-1},x\in\R^d\right\}\geq0$ where $\mathbb{S}^{d-1}=\{x\in\R^d;\,\abs{x}=1\}$. 
		\descitem{(G-3)}{hyp-3} $\varrho_0$ has compact support and furthermore, 
		\begin{equation}
			\mbox{supp}(\varrho_0)\subset \left\{x\in\R^d;\,\inf\limits_{\xi\in\R^d,\abs{\xi}=1}\nabla_x\U_0(x):\xi\otimes\xi\geq\al\right\}\mbox{ for some }\al>0.
		\end{equation}
	    
	\end{description}
    Then there exists $\e_0>0$ such that the following holds: for $\norm{\varrho^{\frac{\gamma-1}{2}}}_{H^{m}(\R^d)}<\e_0$, there exists a global smooth solution $(\varrho,\U)$ of \eqref{eqn-isen-1}--\eqref{eqn-isen-2} corresponding to the initial data $(\varrho_0,\U_0)$ and it satisfies,
    \begin{equation}
    	\left(\varrho^{\frac{\gamma-1}{2}},\U-\V\right)\in C^1([0,\f);H^{m-1}((\R^d))\cap C([0,\f);H^m(\R^d)),
    \end{equation}
where $\V$ is the global solution to \eqref{eqn-Burgers-intro} with initial data $\U_0$.
\end{proposition}
\begin{remark}
	We note that the hypothesis \descref{hyp-3}{(G-3)} restricts support of density $\varrho$ in a set where $\nabla_x\U_0(x):\xi\otimes\xi$ is strictly away from $0$ for all $\xi\in\mathbb{S}^{d-1}$. Therefore, if $\nabla_x u_0\equiv \textbf{0}_d$ or $\nabla_x u_0(x)$ is anti-symmetric matrix for all $x\in\R^d$, then the hypothesis \descref{hyp-3}{(G-3)} forces $\varrho_0\equiv 0$ and in this case the conclusion of Proposition \ref{theorem:global-existence} holds true trivially.
\end{remark}

\subsection{Comparison with previous results}
As we mentioned earlier Grassin \cite{Grassin} proved the global existence of smooth solutions if the initial data satisfies (i) $\inf\left\{\nabla_x\U_0(x):\xi\otimes\xi;\,\xi\in\mathbb{S}^{d-1},x\in\R^d\right\}\geq c_0$ for some $c_0>0$ and (ii) $\varrho_0^{\frac{\gamma-1}{2}}$ is small in Sobolev space $H^m$ for some $m>1+\frac{d}{2}$. When (i) holds, global solution $\overline{\U}$ exists for the Cauchy problem \eqref{eqn-Burgers-intro}. The proof is based on the energy estimate on $\U-\overline{\U}$ and $\varrho^{\frac{\gamma-1}{2}}$. This extends the previous results of Serre \cite{Serre} and Chemin \cite{Chemin}. We also refer to \cite{Serre-exp-vac} where author studied the smooth solution in negative time as well for finite mass and the solution is surrounded by vacuum. Global existence smooth solutions have been shown in \cite{Shkoller-Sideris} when initial velocity is close to a linear function in $\R^d$. In all these articles, it has been crucially used the strict positiveness of $\inf\limits_{\xi\in\mathbb{S}^{d-1},x\in\R^d}\nabla_x\U_0(x):\xi\otimes\xi$. In Theorem \ref{theorem:global-existence}, we relax this positiveness by non-negativeness. We could do this by restricting the support of initial density in the set when $\inf\limits_{\xi\in\mathbb{S}^{d-1}}\nabla_x\U_0(x):\xi\otimes\xi>0$.

Sideris \cite{Sideris-cmp} showed the finite time breakdown of smooth solution when initial data satisfies the integral condition \eqref{condition-Sideris}. As in Theorem \ref{theorem-cont} we show that finite time blow up happens even if we replace the $C^1$ solution by a continuous one. Note that when we talk about continuous solution it is no longer satisfying the equation in the classical way. We consider a weak solution which is continuous and satisfying the entropy inequality \eqref{ineq:entropy}. Theorem \ref{theorem-cont} not only extends the result of Sideris \cite{Sideris-cmp} for continuous entropy solutions, it also tells us that the $C^1$ solution considered in \cite{Sideris-cmp} loses its smoothness in finite time and besides, it fails to be continuous as well. Our result also tells that the continuous solutions constructed in \cite{GK} can not be global if it the data satisfies \eqref{condition-Sideris}.

In the spirit of Sideris \cite{Sideris-cmp}, blow up phenomenon holds for compressible Navier-Stokes \cite{Xin} equations and other hyperbolic systems \cite{Sideris-hyp} as well. When the solution is containing vacuum then blow up of $C^3$ solution has been proved in \cite{Chae-Ha} for irrotational and compressive initial velocity. It has been shown \cite{Sideris-ajm} that for class of data $\varrho_0=\overline{\varrho}+\e\tilde{\varrho}_0,\U_0=\e\tilde{\U}_0$ with small enough $\e>0$ it has been shown the life span of the smooth solution $T^*>C/\e$ for a given $C>0$.

In \cite{BuckShkVic}, authors have shown the shock formation for smooth solution to \eqref{eqn-isen-1}--\eqref{eqn-isen-2} in dimension two at time of order $\e$ when initial data is bounded but it has slope $-1/\e$ for some small enough $\e>0$. They have also shown the stability of the shock formation. In literature, Majda \cite{Majda-stability,Majda-existence} started the study of stability of shock formation.  Christodoulou \cite{Christodoulou} has proved the shock formation for irrotational initial velocity for relativistic Euler equations. We refer to \cite{Christodoulou-Miao} for the non-relativistic case. For 2-D isentropic Euler equations, shock formation has been studied in \cite{LS-Invent} for velocity with vorticity. We note that Theorem \ref{theorem:blow-up-C1-2} and \cite{BuckShkVic} can not be directly compared as \descref{Cond-1}{(ND)} put minimal restriction on initial velocity but Theorem \ref{theorem:blow-up-C1-2} requires the smallness assumption \eqref{smallness-Hm} whereas the results of \cite{BuckShkVic} need large enough negative slope but there is no restriction of type \eqref{smallness-Hm}.

Rest of the article is organized as follows. In the next section, we recall some basic results and notations. We establish energy estimates in section \ref{sec:energy-estimate}. We prove blow up of $C^1$ and continuous solutions in section \ref{sec:C1} and \ref{sec:break-down-cont} respectively. In the last section, we prove the result of global-in-time existence.

\section{Preliminaries}
In this article, we use the following notations.
\begin{itemize}
	\item The notation $\pa_\al^n$ stands for $\frac{d^n}{dx^{\al_1}_1dx^{\al_2}_2\cdots dx^{\al_d}_d}$ where $\al=(\al_1,\cdots,\al_d)$ and $\al_1+\al_2+\dots+\al_d=n$.
	\item For multi index $\al\in\N^n$, $\al=(\al_1,\cdots,\al_n)$, we use $\abs{\al}=\al_1+\cdots+\al_n$.
	\item For $\al=(\al_1,\cdots_n),\B=(\B_1,\cdots,\B_1),\tilde{\B}=(\tilde{\B}_1,\cdots,\tilde{\B}_1)\in\N^n$, we denote $c(\B,\tilde{\B};\al)$ as follows
	\begin{equation}\label{def:c:alpha}
		c(\B,\tilde{\B};\al)=\frac{\al_1! \al_2!\cdots\al_n!}{\B_1!\B_2!\cdots\B_n!\tilde{\B}_1!\tilde{\B}_2!\cdots\tilde{\B}_n!}\mbox{ where }\al_i=\B_i+\tilde{\B}_i\mbox{ for }1\leq i\leq n.
	\end{equation}
	\item We often use the following inequality which we refer as ``$\e$-inequality"
	\begin{equation}\label{inequality-eps}
		ab\leq \e a^2+\frac{1}{4\e}b^2\mbox{ for }a,b\in\R\mbox{ and }\e>0.
	\end{equation}
    \item Let $p>1,q>1$ and $a,b\geq0$, Then the following inequality holds,
    \begin{equation}\label{ineq:Young}
    	ab\leq \frac{a^p}{p}+\frac{b^q}{q}\quad\mbox{ if $p,q$ satisfy }\frac{1}{p}+\frac{1}{q}=1.
    \end{equation}
    This is known as `Young's inequality'. 
	\item By $W^{m,p}(\R^d)$ we denote the standard Sobolev spaces for $m\in \mathbb{N}$ and $p\geq1$. We use the notation $H^m=W^{m,2}$.
\end{itemize}

In the following lemma, we recall the equivalence of various norms in finite dimension.
\begin{lemma}\label{lemma:linear-algebra-1}
	Let $m\geq1,n\geq 1$ and $d\geq 1$. There exists $C_{m,n,d}>1$ such that the following holds,
	\begin{equation}\label{equiv-norm-linear-alg}
	\frac{1}{C_{m,n,d}}\sum\limits_{k=1}^{d}\abs{x_k}^n\leq 	\left(\sum\limits_{k=1}^{d}\abs{x_k}^m\right)^{\frac{n}{m}}\leq C_{m,n,d}\sum\limits_{k=1}^{d}\abs{x_k}^n\mbox{ for all }x=(x_1,\cdots,x_d)\in\R^d.
	\end{equation}
\end{lemma}
For proof of Lemma \ref{lemma:linear-algebra-1} we refer to \cite{Lax:linear-alg}. Next, we recall a classical result on embedding of Sobolev spaces.
\begin{lemma}\label{lemma:Morrey}[Morrey]
	Let $q>N$. Then we have the following continuous injection $W^{1,q}(\R^N)\subset L^\f(\R^N)$. Furthermore, for all $g\in W^{1,q}(\R^N)$,
	\begin{equation}
		\abs{g(z_1)-g(z_2)}\leq C\abs{z_1-z_2}^{\al}\norm{\nabla_x g}_{L^q(\R^N)}\mbox{ for a.e. }z_1,z_2\in\R^N,
	\end{equation}
where $C=C(q,N)$, $\al=1-\frac{1}{q}$ and it does not depend on $g$.
\end{lemma}
For proof of Lemma \ref{lemma:Morrey} we refer to \cite[Theorem 9.12]{Brezis-book}.

\section{Energy estimates}\label{sec:energy-estimate}
The blow up of $C^1$ solution for a class of initial condition relies on an energy estimate in the spirit of \cite{Kato}. Here, we reproduce the regularity result of Kato \cite{Kato} for integer exponents. 
\begin{proposition}\label{proposition:reg}
	Let $(\varrho,\U)\in C^1([0,T^*)\times\R^d)$ be a solution of isentropic Euler system \eqref{eqn-isen-1}--\eqref{eqn-isen-2} corresponding to the initial data $(\varrho_0,\U_0)$. We further assume that $(\nabla_x^k\varrho_0^{\frac{\ga-1}{2}}, \nabla_x^k\U_0)\in H^m(\R^d)$ for some $m\in\mathbb{N}$ and $k\in\mathbb{N}\cup\{0\}$. Then, we have
	\begin{equation}
		(\nabla_x^k\varrho^{\frac{\ga-1}{2}}(t,\cdot), \nabla_x^k\U(t,\cdot))\in H^m(\R^d)\mbox{ for all }t\in[0,T^*).
	\end{equation}
\end{proposition}
We note that the exponent of regularity in Proposition \ref{proposition:reg} are considered to be positive integers. We expect that it can be extended to the positive real numbers in a similar spirit as in \cite{Kato}.

The proof of Proposition \ref{proposition:reg} can be done in a similar manner as in \cite{Kato}. For sake of completeness, we provide here some explicit estimates which will help us in section \ref{sec:C1}. To prove the energy estimates we need the following results on interpolation of $W^{1,\f}$ and $W^{3,2}$. We refer to \cite{Adams} for more on interpolation of spaces.
\begin{lemma}\label{lemma:interpol-1}
	Let $\psi\in C^3_c(\R^d)$. Then the following holds,
	\begin{equation}\label{interpolation-inequality-1}
		\int\limits_{\R^d}\abs{\nabla_x^2\psi}^4\,dx\leq C_d\norm{\nabla_x\psi}^2_{L^\f(\R^d)}\int\limits_{\R^d}\abs{\nabla_x^3\psi}^2\,dx.
	\end{equation}
\end{lemma}
\begin{proof}
	Throughout the proof we use $C_d$ as a generic constant depending only on dimension $d$. By using integration by parts, we can see that
	\begin{align*}
		\sum\limits_{i,j=1}^{d}\int\limits_{\R^d}\abs{\pa^2_{ij}\psi}^4\,dx&=\sum\limits_{i,j=1}^{d}\int\limits_{\R^d}\pa^2_{ij}\psi(\pa^2_{ij}\psi)^3\,dx\\
		&=-C_d\sum\limits_{i,j=1}^{d}\int\limits_{\R^d}\pa_{i}\psi(3\pa^2_{ij}\psi)^2\pa^3_{ijj}\psi\,dx.
	\end{align*}
Applying $\e$-inequality \eqref{inequality-eps} with $\e=1/2$ we obtain,
	\begin{align*}		\\
			\sum\limits_{i,j=1}^{d}\int\limits_{\R^d}\abs{\pa^2_{ij}\psi}^4\,dx&\leq \frac{1}{2}	\sum\limits_{i,j=1}^{d}\int\limits_{\R^d}\abs{\pa^2_{ij}\psi}^4\,dx+C_d	\sum\limits_{i,j=1}^{d}\int\limits_{\R^d}\norm{\pa_{i}\psi}_{L^\f(\R^d)}^2\abs{\pa^3_{ijj}\psi}^2\,dx\\
		&\leq \frac{1}{2}	\sum\limits_{i,j=1}^{d}\int\limits_{\R^d}\abs{\pa^2_{ij}\psi}^4\,dx+C_d	\norm{\nabla_x\psi}^2_{L^\f(\R^d)}\int\limits_{\R^d}\abs{\nabla_x^3\psi}^2\,dx.
	\end{align*}
	Hence, we have
	\begin{equation*}
		\sum\limits_{i,j=1}^{d}\int\limits_{\R^d}\abs{\pa^2_{ij}\psi}^4\,dx\leq C_d	\norm{\nabla_x\psi}^2_{L^\f(\R^d)}\int\limits_{\R^d}\abs{\nabla_x^3\psi}^2\,dx.
	\end{equation*}
	By using the inequality \eqref{equiv-norm-linear-alg} we conclude the result \eqref{interpolation-inequality-1}.
\end{proof}

In the next result, we estimate $\norm{\nabla_x^2\phi\nabla_x^3\psi}_{L^2}$ in terms of $\norm{\nabla_x^4\psi}_{L^2},\norm{\nabla_x^4\phi}_{L^2}$ and the $W^{1,\f}$ norm of $\psi,\phi$. It plays an important role in derivation of energy estimates.
\begin{lemma}\label{lemma:interpol-2} Let $\phi,\psi\in C^4_{c}(\R^d),\,d\geq1$. Then the following holds,
	\begin{align}
		&\int\limits_{\R^d}\left[\abs{\nabla_x^2\psi}^2\abs{\nabla_x^3\psi}^2+\abs{\nabla_x^2\phi}^2\abs{\nabla_x^3\phi}^2\right]\,dx+\int\limits_{\R^d}\left[\abs{\nabla_x^2\psi}^2\abs{\nabla_x^3\phi}^2+\abs{\nabla_x^2\phi}^2\abs{\nabla_x^3\psi}^2\right]\,dx\nonumber\\
		&\leq C_{d}\left(1+\norm{\nabla_x\phi}_{L^\f(\R^d)}+\norm{\nabla_x\psi}_{L^\f(\R^d)}\right)^3\int\limits_{\R^d}\abs{\nabla_x^4\phi}^2+\abs{\nabla_x^4\psi}^2\,dx.\label{ineq-lemma-2}
	\end{align}
\end{lemma}
\begin{proof} 
	Applying integration by part, we get
	\begin{align}
		\sum\limits_{i,j,k,l,r=1}^{d}\int\limits_{\R^d}\abs{\pa^2_{ij}\psi}^2\abs{\pa^3_{klr}\phi}^2\,dx
		&=-	\sum\limits_{i,j,k,l,r=1}^{d}\int\limits_{\R^d}\pa_j\psi\pa^3_{iij}\psi\abs{\pa^3_{klr}\phi}^2\,dx\nonumber\\
		&-2	\sum\limits_{i,j,k,l,r=1}^{d}\int\limits_{\R^d}\pa_j\psi\pa^2_{ij}\psi\pa^3_{klr}\phi\pa^4_{iklr}\phi\,dx.\label{ineq-1}
	\end{align}
	By using the relation between finite dimensional norms \eqref{equiv-norm-linear-alg}, we see
	\begin{equation}
		\frac{1}{C_{d}}\int\limits_{\R^d}\left(\sum\limits_{i,j,k=1}^{d}(\pa^3_{ijk}\psi)^2\right)^{3/2}\,dx\leq \int\limits_{\R^d}\abs{\nabla_x^3\psi}^3\,dx\leq {C_{d}}\int\limits_{\R^d}\left(\sum\limits_{i,j,k=1}^{d}(\pa^3_{ijk}\psi)^2\right)^{3/2}\,dx, \label{equiv-norm-1}
	\end{equation}
    where $C_d$ is a generic constant depending only on dimension $d$. With the help of \eqref{equiv-norm-1}, Young's inequality \eqref{ineq:Young} and $\e$-inequality, from \eqref{ineq-1}, we obtain
	\begin{align}
		\sum\limits_{i,j,k,l,r=1}^{d}\int\limits_{\R^d}\abs{\pa^2_{ij}\psi}^2\abs{\pa^3_{klr}\phi}^2\,dx\leq &C_d\left(1+\norm{\nabla_x\psi}_{L^\f(\R^d)}\right)^2  \int\limits_{\R^d}\left[\abs{\nabla_x^3\psi}^3+\abs{\nabla_x^3\phi}^3\right]\,dx\nonumber\\
		&+C_d\left(1+\norm{\nabla_x\psi}_{L^\f(\R^d)}\right)^2\int\limits_{\R^d}\abs{\nabla_x^4\psi}^2\,dx.\label{ineq-2}
	\end{align}
	Using integration by parts, we have,
	\begin{align}
		&\int\limits_{\R^d}\left(\sum\limits_{i,j,k=1}^{d}(\pa^3_{ijk}\psi)^2\right)^{3/2}\,dx\nonumber\\
		&=\sum\limits_{i^\p,j^\p,k^\p=1}^{d}\int\limits_{\R^d}\left(\sum\limits_{i,j,k=1}^{d}(\pa^3_{ijk}\psi)^2\right)^{1/2}(\pa^3_{i^\p j^\p k^\p}\psi)^2\,dx\nonumber\\
		&=-\sum\limits_{i^\p,j^\p,k^\p=1}^{d}\int\limits_{\R^d}\left(\sum\limits_{i,j,k=1}^{d}(\pa^3_{ijk}\psi)^2\right)^{1/2}\pa^2_{ j^\p k^\p}\psi\pa^4_{i^\p i^\p j^\p k^\p}\psi\,dx\nonumber\\
		&-\sum\limits_{i^\p,j^\p,k^\p=1}^{d}\sum\limits_{i,j,k=1}^{d}\int\limits_{\R^d}\left(\sum\limits_{i,j,k=1}^{d}(\pa^3_{ijk}\psi)^2\right)^{-1/2}\pa^2_{ j^\p k^\p}\psi\pa^3_{i^\p j^\p k^\p}\psi\pa^4_{i^\p i j k}\psi\pa^3_{ijk}\psi\,dx.
	\end{align}
	Now, with the help of $\e$-inequality \eqref{inequality-eps}, we get
	\begin{align}
		&\int\limits_{\R^d}\left(\sum\limits_{i,j,k=1}^{d}(\pa^3_{ijk}\psi)^2\right)^{3/2}\,dx\nonumber\\
		&\leq \e \sum\limits_{i^\p,j^\p,k^\p=1}^{d}\sum\limits_{i,j,k=1}^{d}\int\limits_{\R^d}(\pa^3_{ijk}\psi)^2 (\pa^2_{ j^\p k^\p}\psi)^2\,dx+C_\e \sum\limits_{i^\p,j^\p,k^\p=1}^{d}\int\limits_{\R^d}\abs{\pa^4_{i^\p i^\p j^\p k^\p}\psi}^2\,dx\nonumber\\
		&+C_{d}\sum\limits_{i^\p,j^\p,k^\p,i=1}^{d}\int\limits_{\R^d}\left(\sum\limits_{i,j,k=1}^{d}(\pa^3_{ijk}\psi)^2\right)^{1/2}\abs{\pa^2_{ j^\p k^\p}\psi}\abs{\pa^4_{i^\p i j k}\psi}\,dx.
	\end{align}
	Again by similar argument, we have
	\begin{align}
		&\int\limits_{\R^d}\left(\sum\limits_{i,j,k=1}^{d}(\pa^3_{ijk}\psi)^2\right)^{3/2}\,dx\nonumber\\
		&\leq \e \sum\limits_{i^\p,j^\p,k^\p=1}^{d}\sum\limits_{i,j,k=1}^{d}\int\limits_{\R^d}(\pa^3_{ijk}\psi)^2 (\pa^2_{ j^\p k^\p}\psi)^2\,dx+C_\e \sum\limits_{i^\p,j^\p,k^\p=1}^{d}\int\limits_{\R^d}\abs{\pa^4_{i^\p i^\p j^\p k^\p}\psi}^2\,dx\nonumber\\
		&+\e C_{d} \int\limits_{\R^d}\left(\sum\limits_{i,j,k=1}^{d}(\pa^3_{ijk}\psi)^2\right)\left(\sum\limits_{i,j=1}^{d}(\pa^2_{ij}\psi)^2\right)\,dx+C_{\e,d}\int\limits_{\R^d} \abs{\nabla^4_x\psi}\,dx,\nonumber
	\end{align}
	or equivalently,
	\begin{equation}\label{ineq-psi}
		\int\limits_{\R^d}\abs{\nabla_x^3\psi}^3\,dx\leq \e C_{d}\int\limits_{\R^d}\abs{\nabla^3_x\psi}^2\abs{\nabla^2_x\psi}^2\,dx+C_\e C_{d}\int\limits_{\R^d}\abs{\nabla_x^4\psi}^2\,dx.
	\end{equation}
	Similarly, we have
	\begin{equation}\label{ineq-phi}
		\int\limits_{\R^d}\abs{\nabla_x^3\phi}^3\,dx\leq \e C_{d}\int\limits_{\R^d}\abs{\nabla^3_x\phi}^2\abs{\nabla^2_x\phi}^2\,dx+C_\e C_{d}\int\limits_{\R^d}\abs{\nabla_x^4\phi}^2\,dx.
	\end{equation}
	Let's define $\mathcal{I}_{i},\,1\leq i\leq4$, as follows,
	\begin{align*}
		\mathcal{I}_1&:=\sum\limits_{i,j,k,l,r=1}^{d}\int\limits_{\R^d}\abs{\pa^2_{ij}\psi}^2\abs{\pa^3_{klr}\phi}^2\,dx,\quad\quad \mathcal{I}_2:=\sum\limits_{i,j,k,l,r=1}^{d}\int\limits_{\R^d}\abs{\pa^2_{ij}\phi}^2\abs{\pa^3_{klr}\psi}^2\,dx,\\
		\mathcal{I}_3&:=\sum\limits_{i,j,k,l,r=1}^{d}\int\limits_{\R^d}\abs{\pa^2_{ij}\psi}^2\abs{\pa^3_{klr}\psi}^2\,dx,\quad\quad	\mathcal{I}_4:=\sum\limits_{i,j,k,l,r=1}^{d}\int\limits_{\R^d}\abs{\pa^2_{ij}\phi}^2\abs{\pa^3_{klr}\phi}^2\,dx.
	\end{align*}
	We also set $M:=1+\norm{\nabla_x\psi}_{L^\f(\R^d)}+\norm{\nabla_x\phi}_{L^\f(\R^d)}$. Combining \eqref{ineq-psi} and \eqref{ineq-phi} with \eqref{ineq-2}, we obtain
	\begin{equation}\label{est-I1}
		\mathcal{I}_1\leq \e C_{d}M^2\left[\mathcal{I}_3+\mathcal{I}_4\right]+C_{\e}C_{d}M^2\int\limits_{\R^d}\left[\abs{\nabla_x^4\psi}^2+\abs{\nabla_x^4\phi}^2\right]\,dx.
	\end{equation}
	By a similar argument, we have
	\begin{equation}\label{est-I2}
		\mathcal{I}_k\leq \e C_{d}M^2\left[\mathcal{I}_3+\mathcal{I}_4\right]+C_{\e}C_{d}M^2\int\limits_{\R^d}\left[\abs{\nabla_x^4\psi}^2+\abs{\nabla_x^4\phi}^2\right]\,dx\mbox{ for }k=2,3,4.
	\end{equation}
	Finally, we combine \eqref{est-I1}, \eqref{est-I2} and then choosing appropriate $\e$, we obtain \eqref{ineq-lemma-2}.
\end{proof}
\subsection{Proof of Proposition \ref{proposition:reg}}
With the help of Lemma \ref{lemma:interpol-1} and \ref{lemma:interpol-2}, we are going to prove Proposition \ref{proposition:reg}. Proof also relies on the symmetrization of \eqref{eqn-isen-1}--\eqref{eqn-isen-2} in the spirit of \cite{KMU}. The method is similar to the one in \cite{Kato}. Here, we can reduce the regularity exponent even smaller than $1+\frac{d}{2}$ due to our assumption on $C^1$ of $\varrho$ and $\U$.
\begin{proof}[Proof of Proposition \ref{proposition:reg}]

We consider the following symmetrization of \eqref{eqn-isen-1}--\eqref{eqn-isen-2} in the spirit of \cite{KMU}
\begin{align}
	\pa_t\pi+\U\cdot\nabla_x\pi+C_1\pi\dv_x\U&=0,\\
	\pa_t\U+\U\cdot\nabla_x\U+C_1\pi\nabla_x\pi&=0,
\end{align}	
where $C_1=\frac{\gamma-1}{2}$ and $\pi=\sqrt{\frac{\gamma-1}{4\gamma}}\varrho^\frac{\gamma-1}{2}$. Applying derivative $\pa_\al^n$ for $\al=(\al_1,\cdots,\al_d)$ with $\al_1+\dots+\al_d=n$ we obtain,
\begin{align}
	\pa_t\pa^{n}_\al\pi+\sum\limits_{\B_1+\B_2=\al}c(\B_1,\B_2;\al)\pa^{\abs{\B_1}}_{\B_1}\U\cdot\nabla_x\pa^{\abs{\B_2}}_{\B_2}\pi+C_1\sum\limits_{\B_1+\B_2=\al}c(\B_1,\B_2;\al)\pa^{\abs{\B_1}}_{\B_1}\pi\dv_x\pa^{\abs{\B_2}}_{\B_2}\U&=0,\label{eq-dn-pi}\\
	\pa_t\pa^{n}_{\al}\U+\sum\limits_{\B_1+\B_2=\al}c(\B_1,\B_2;\al)\pa^{\abs{\B_1}}_{\B_1}\U\cdot\nabla_x\pa^{\abs{\B_2}}_{\B_2}\U+C_1\sum\limits_{\B_1+\B_2=\al}c(\B_1,\B_2;\al)\pa^{\abs{\B_1}}_{\B_1}\pi\nabla_x\pa^{\abs{\B_2}}_{\B_2}\pi&=0,\label{eq-dn-u}
\end{align}	
where $c(\B_1,\B_2;\al)$ is as in \eqref{def:c:alpha}. We rearrange \eqref{eq-dn-pi} as follows
\begin{align}
	\pa_t\pa^{n}_\al\pi
	&=-(\U\cdot\nabla_x\pa^{n}_{\al}\pi+\pa^{n}_{\al}\U\cdot\nabla_x\pi)-\sum\limits_{\B_1+\B_2=\al,\abs{\B_1}\geq1,\abs{\B_2}\geq1}c(\B_1,\B_2;\al)\pa^{\abs{\B_1}}_{\B_1}\U\cdot\nabla_x\pa^{\abs{\B_2}}_{\B_2}\pi\nonumber\\
	&-C_1(\pi\dv_x\pa^{n}_{\al}\U+\pa^{n}_{\al}\pi\dv_x\U)-C_1\sum\limits_{\B_1+\B_2=\al,\abs{\B_1}\geq1,\abs{\B_2}\geq1}c(\B_1,\B_2;\al)\pa^{\abs{\B_1}}_{\B_1}\pi\dv_x\pa^{\abs{\B_2}}_{\B_2}\U.\label{eq:dn-pi-1}
\end{align}
Multiplying \eqref{eq:dn-pi-1} with $\pa^n_\al\pi$ and then integrating over $\R^d$ we get
\begin{align}
	\frac{d}{dt}\int\limits_{\R^d}\abs{\pa^n_{\al}\pi}^2&\leq C_{d,n}\left(1+\norm{\nabla_x\U(t,\cdot)}_{L^\f(\R^d)}+\norm{\nabla_x\pi(t,\cdot)}_{L^\f(\R^d)}\right)\int\limits_{\R^d}\left[\abs{\nabla_x^n\pi}^2+\abs{\nabla_x^n\U}^2\right]\,dx\nonumber\\\
	&+C_{d,n}\sum\limits_{k=1}^{n-1}\int\limits_{\R^d}\left[\abs{\nabla_x^k\pi}^2\abs{\nabla_x^{n-k+1}\U}^2\,dx+\abs{\nabla_x^k\U}^2\abs{\nabla_x^{n-k+1}\pi}^2\right]\,dx-C_1\int\limits_{\R^d}\pi\dv_x\pa^n_{\al}\U\pa^n_{\al}\pi.\label{eq:dn-pi-2}
\end{align}
We arrange \eqref{eq-dn-u} to have
\begin{align}
		\pa_t\pa^{n}_{\al}\U&=-(\U\cdot\nabla_x\pa^{n}_{\al}\U+\pa^{n}_{\al}\U\cdot\nabla_x\U)-\sum\limits_{\B_1+\B_2=\al,\abs{\B_1}\geq1,\abs{\B_2}\geq1}c(\B_1,\B_2;\al)\pa^{\abs{\B_1}}_{\B_1}\U\cdot\nabla_x\pa^{\abs{\B_2}}_{\B_2}\U\nonumber\\
	&-C_1(\pi\nabla_x\pa^{n}_{\al}\pi+\pa^{n}_{\al}\pi\nabla_x\pi)-C_1\sum\limits_{\B_1+\B_2=\al,\abs{\B_1}\geq1,\abs{\B_2}\geq1}c(\B_1,\B_2;\al)\pa^{\abs{\B_1}}_{\B_1}\pi\nabla_x\pa^{\abs{\B_2}}_{\B_2}\pi,\label{eq:dn-u-1}
\end{align}
where $c(\al_1,\B_2;\al)$ is as in \eqref{def:c:alpha}. Multiplying \eqref{eq:dn-u-1} with $\pa^n_\al\U$ and then integrating over $\R^d$ we obtain
\begin{align}
	\frac{d}{dt}\int\limits_{\R^d}\abs{\pa^n_{\al}\U}^2&\leq C_{d,n}\left(1+\norm{\nabla_x\U(t,\cdot)}_{L^\f(\R^d)}+\norm{\nabla_x\pi(t,\cdot)}_{L^\f(\R^d)}\right)\int\limits_{\R^d}\left[\abs{\nabla_x^n\pi}^2+\abs{\nabla_x^n\U}^2\right]\,dx\nonumber\\
	&+C_{d,n}\sum\limits_{k=1}^{n-1}\int\limits_{\R^d}\left[\abs{\nabla_x^k\pi}^2\abs{\nabla_x^{n-k+1}\pi}^2\,dx+\abs{\nabla_x^k\U}^2\abs{\nabla_x^{n-k+1}\U}^2\right]\,dx-C_1\int\limits_{\R^d}\pi\nabla_x\pa^n_{\al}\pi\cdot\pa^n_{\al}\U.
\end{align}
Applying integration by parts we get
\begin{align}
	\frac{d}{dt}\int\limits_{\R^d}\abs{\pa^n_{\al}\U}^2&\leq C_{d,n}\left(1+\norm{\nabla_x\U(t,\cdot)}_{L^\f(\R^d)}+\norm{\nabla_x\pi(t,\cdot)}_{L^\f(\R^d)}\right)\int\limits_{\R^d}\left[\abs{\nabla_x^n\pi}^2+\abs{\nabla_x^n\U}^2\right]\,dx\nonumber\\
	&+C_{d,n}\sum\limits_{k=1}^{n-1}\int\limits_{\R^d}\left[\abs{\nabla_x^k\pi}^2\abs{\nabla_x^{n-k+1}\pi}^2\,dx+\abs{\nabla_x^k\U}^2\abs{\nabla_x^{n-k+1}\U}^2\right]\,dx\nonumber\\
	&+C_1\int\limits_{\R^d}\pa^n_{\al}\pi\nabla_x\pi\cdot\pa^n_{\al}\U\,dx+C_1\int\limits_{\R^d}\pi\dv_x\pa^n_{\al}\U\pa^n_{\al}\pi\,dx.\nonumber
\end{align}
We further estimate,
\begin{align}
	\frac{d}{dt}\int\limits_{\R^d}\abs{\pa^n_{\al}\U}^2&\leq  C_{d,n}\left(1+\norm{\nabla_x\U(t,\cdot)}_{L^\f(\R^d)}+\norm{\nabla_x\pi(t,\cdot)}_{L^\f(\R^d)}\right)\int\limits_{\R^d}\left[\abs{\nabla_x^n\pi}^2+\abs{\nabla_x^n\U}^2\right]\,dx\nonumber\\
	&+C_{d,n}\sum\limits_{k=1}^{n-1}\int\limits_{\R^d}\left[\abs{\nabla_x^k\pi}^2\abs{\nabla_x^{n-k+1}\pi}^2\,dx+\abs{\nabla_x^k\U}^2\abs{\nabla_x^{n-k+1}\U}^2\right]\,dx\nonumber\\
	&+C_1\int\limits_{\R^d}\pi\pa^n_{\al}\pi\dv_x\pa^n_{\al}\U\,dx.\label{eq:dn-u-2}
\end{align}
Combining \eqref{eq:dn-pi-2}, \eqref{eq:dn-u-2} along with the help of \eqref{equiv-norm-linear-alg} we have
\begin{align}
	&\frac{d}{dt}\int\limits_{\R^d}\left[\abs{\nabla_x^n\U}^2+\abs{\nabla_x^n\pi}^2\right]\,dx\nonumber\\
	&\leq C_{d,n}\left(1+\norm{\nabla_x\U(t,\cdot)}_{L^\f(\R^d)}+\norm{\nabla_x\pi(t,\cdot)}_{L^\f(\R^d)}\right)\int\limits_{\R^d}\left[\abs{\nabla_x^n\pi}^2+\abs{\nabla_x^n\U}^2\right]\,dx\nonumber\\
	&+C_{d,n}\sum\limits_{k=1}^{n-1}\int\limits_{\R^d}\left[\abs{\nabla_x^k\pi}^2\abs{\nabla_x^{n-k+1}\pi}^2\,dx+\abs{\nabla_x^k\U}^2\abs{\nabla_x^{n-k+1}\U}^2\right]\,dx\nonumber\\
	&+C_{d,n}\sum\limits_{k=1}^{n-1}\int\limits_{\R^d}\left[\abs{\nabla_x^k\pi}^2\abs{\nabla_x^{n-k+1}\U}^2\,dx+\abs{\nabla_x^k\U}^2\abs{\nabla_x^{n-k+1}\pi}^2\right]\,dx.
\end{align}
By using Lemma \ref{lemma:interpol-1} and \ref{lemma:interpol-2} with appropriate choice of $\psi,\phi$, we obtain
\begin{align}
	&\frac{d}{dt}\int\limits_{\R^d}\left[\abs{\nabla_x^n\U}^2+\abs{\nabla_x^n\pi}^2\right]\,dx\nonumber\\
	&\leq C_{d,n}\left(1+\norm{\nabla_x\U(t,\cdot)}_{L^\f(\R^d)}+\norm{\nabla_x\pi(t,\cdot)}_{L^\f(\R^d)}\right)\int\limits_{\R^d}\left[\abs{\nabla_x^n\pi}^2+\abs{\nabla_x^n\U}^2\right]\,dx.\nonumber
\end{align}
By using Gr\"onwall's inequality, we get
\begin{equation}\label{energy-est-1}
	\int\limits_{\R^d}\left[\abs{\nabla_x^n\U}^2+\abs{\nabla_x^n\pi}^2\right]\,dx\leq e^{M(t)}\int\limits_{\R^d}\left[\abs{\nabla_x^n\U_0}^2+\abs{\nabla_x^n\pi_0}^2\right]\,dx,
\end{equation}
where $M(t)$ is defined as
\begin{equation}
	 M(t):=\int\limits_{0}^{t}C_{d,n}(1+\norm{\nabla_x\U(s,\cdot)}_{L^\f(\R^d)}+\norm{\nabla_x\pi(s,\cdot)}_{L^\f(\R^d)})\,ds.
\end{equation}
This completes the proof of Proposition \ref{proposition:reg}.
	
\end{proof}
\section{Blow up of $C^1$ solutions}\label{sec:C1}

\subsection{Proof of Proposition \ref{theorem:blow-up-C1-1}}
In this subsection, we prove the blow up result of Proposition \ref{theorem:blow-up-C1-1}. Here, we do not derive the Riccati equation directly. Rather, we compare the velocity $\U(t,\cdot)$ with the solution of vectorial Burgers equation \eqref{eqn-Burgers-intro} with initial data $\U_0$.
\begin{proof}[Proof Proposition \ref{theorem:blow-up-C1-1}:]
	We argue by contradiction. Suppose that $T^*>\frac{1}{\la_{max}}$. Let us consider $\V$ to be the solution to the following Cauchy problem,
	\begin{align}
		\pa_t\V+\V\cdot\nabla_x\V&=0,\mbox{ for }x\in\R^d,t>0,\label{system-V}\\
		\V(0,x)&=\U_0(x)\mbox{ for }x\in\R^d.\label{data:V}
    \end{align}
  Set $\pi:=\sqrt{\frac{\gamma-1}{4\gamma}}\varrho^{\frac{\ga-1}{2}}$ and we consider the following symmetrization of the system \eqref{eqn-isen-1}--\eqref{eqn-isen-2},
	\begin{align}
		\pa_t\pi+\U\cdot\nabla_x\pi+C_1\pi\dv_x\U&=0,\label{system-pi}\\
		\pa_t\U+\U\cdot\nabla_x\U+C_1\pi\nabla_x\pi&=0,\label{system-U}
	\end{align}	
where $C_1=\frac{\gamma-1}{2}$. From \eqref{system-V} and \eqref{system-U} we have
\begin{equation*}
		\pa_t(\U-\V)+\U\cdot\nabla_x\U-\V\cdot\nabla_x\V+C_1\pi\nabla_x\pi=0.
\end{equation*}	
Let $\boldsymbol{\varphi}=\U-\V$. Then, we get
\begin{equation*}
	\pa_t\boldsymbol{\varphi}+\boldsymbol{\varphi}\cdot\nabla\U+\V\cdot\nabla_x\boldsymbol{\varphi}+C_1\pi\nabla_x\pi=0.
\end{equation*}
Multiply by $\boldsymbol{\varphi}$ to get
\begin{equation}
	\frac{1}{2}\pa_t\abs{\boldsymbol{\varphi}}^2+\nabla\U:\boldsymbol{\varphi}\otimes\boldsymbol{\varphi}+\frac{1}{2}\V\cdot\nabla_x\abs{\boldsymbol{\varphi}}^2+C_1\pi\nabla_x\pi\cdot\boldsymbol{\varphi}=0,
\end{equation}
or equivalently,
\begin{equation}\label{eq:ODE-varphi}
	(\pa_t+\V\cdot\nabla_x)\left(\frac{\abs{\boldsymbol{\varphi}}^2}{2}\right)=-\nabla\U:\boldsymbol{\varphi}\otimes\boldsymbol{\varphi}-C_1\pi\nabla_x\pi\cdot\boldsymbol{\varphi}.
\end{equation}
Now we consider $X(t,x_0)$ defined as 
\begin{equation*}
 \frac{d X(t,x_0)}{dt}=\V(t,X(t,x_0))\mbox{ and }X(0,x_0)=x_0.
\end{equation*}
We define $g(t)=\frac{1}{2}\abs{\boldsymbol{\varphi}(t,X(t,x_0))}^2$ to get
\begin{align*}
	\frac{d g(t)}{dt}&=-\nabla_x\U(t,X(t,x_0)):\boldsymbol{\varphi}(t,X(t,x_0))\otimes \boldsymbol{\varphi}(t,X(t,x_0))\\
	&-C_1\pi(t,X(t,x_0))\nabla_x\pi(t,X(t,x_0))\cdot \boldsymbol{\varphi}(t,X(t,x_0))\\
	&\leq 2\abs{\nabla_x\U(t,X(t,x_0))} a(t)+C_1\pi(t,X(t,x_0))\abs{\nabla_x\pi(t,X(t,x_0))}\sqrt{g(t)}.
\end{align*}
Since $T^*>1/\la_{max}$, we have $M\geq \norm{\nabla_x\U}_{L^\f([0,1/\la_{max}]\times\R^d)}$ and from \eqref{condition-eps-1} it follows that \\$\e_0\geq \norm{\pi\nabla_x\pi}_{L^\f([0,1/\la_{max}]\times\R^d)}$. Then, we can write
\begin{equation*}
	g^\p(t)\leq Mg(t)+\e_0 \sqrt{g(t)}.
\end{equation*}
Let $f(t)=\sqrt{g}(t)$. We have
\begin{equation*}
	f^\p(t)\leq M f(t)+\e_0\mbox{ or equivalently, }\left(f(t)e^{-Mt}\right)^\p\leq\e_0 e^{-Mt}.
\end{equation*}
Subsequently, we get
\begin{equation*}
	f(t)\leq e^{Mt}\left[f(0)+\frac{\e_0}{M}\left(1-e^{Mt}\right)\right]=\frac{\e_0}{M}\left(e^{Mt}-1\right).
\end{equation*}
From the assumption on $\U_0$ we have $\U_0(x_0+r\xi)-\U_0(x_0)=-\la_0r\xi $. Let $x_{\xi}=x_0+r\xi$, $X(t)=x_\xi+t\U_0(x_\xi)$ and $Y(t)=x_0+t\U_0(x_0)$. Then we have
\begin{align*}
	\abs{X(t)-Y(t)}^2&=\abs{x_\xi-x_0+t(\U_0(x_\xi)-\U_0(x_0))}^2\\
	&=\abs{x_\xi-x_0}^2+2t(\U_0(x_\xi)-\U_0(x_0))\cdot(x_\xi-x_0)+t^2\abs{\U_0(x_\xi)-\U_0(x_0)}^2\\
	&= \abs{x_\xi-x_0}^2-2\la_0t\abs{x_\xi-x_0}^2+\la_0^2t^2\abs{x_\xi-x_0}^2\\
	&=(1-t\la_0)^2\abs{x_\xi-x_0}^2.
\end{align*}
Since $\V=\U-\boldsymbol{\varphi}$, we note that
\begin{align*}
	&(x_\xi-x_0)\cdot(\U_0(x_\xi)-\U_0(x_0))\\
	&=(x_\xi-x_0)\cdot(\V(X(t))-\V(Y(t)))\\
	&=(x_\xi-x_0)\cdot(\U(X(t))-\U(Y(t)))-(x_\xi-x_0)\cdot(\boldsymbol{\varphi}(X(t))-\boldsymbol{\varphi}(Y(t))).
\end{align*}
As $\norm{\boldsymbol{\varphi}(t)}_{L^\f(\R^d)}\leq e^{Mt}\e_0$ for all $t\in[0,1/\la_{max}]$, we get
\begin{equation}
	\abs{(x_\xi-x_0)\cdot(\boldsymbol{\varphi}(X(t))-\boldsymbol{\varphi}(Y(t)))}\leq \frac{2e^{Mt}}{M}\abs{x_\xi-x_0}\e_0\leq\frac{2e^{\frac{M}{\la_{max}}}}{M}r\e_0.
\end{equation}
Similarly, we have
\begin{align*}
	&\abs{(x_\xi-x_0)\cdot(\U(X(t))-\U(Y(t)))}\\
	&=\abs{\int\limits_{0}^{1}(X(t)-Y(t))\cdot\nabla_x\U(t,\theta X(t)+(1-\theta)Y(t))\cdot(x_\xi-x_0)\,d\theta}\\
	&\leq \abs{X(t)-Y(t)}\abs{x_\xi-x_0}\norm{\nabla_x\U(t,\cdot)}_{L^\f(\R^d)}\\
	&=(1-\la_0 t)\abs{x_\xi-x_0}^2\norm{\nabla_x\U(t,\cdot)}_{L^\f(\R^d)}\\
	&=(1-\la_0t)r^2M.
\end{align*}
Subsequently, 
\begin{align*}
	\la_0r^2=\la_0\abs{x_\xi-x_0}^2&=\abs{(x_\xi-x_0)\cdot(\V_0(x_\xi)-\V_0(x_0))}\\
	&\leq (1-\la_0t)r^2M+\frac{2e^{\frac{M}{\la_{max}}}}{M}r\e_0.
\end{align*}
Therefore, for $t=\frac{1}{\la_{max}}$ we obtain
\begin{align*}
	\la_0r\leq \left(1-\frac{\la_0}{\la_{max}}\right)rM+\frac{2e^{\frac{M}{\la_{max}}}}{M}\e_0&= \frac{\la_{max}-\la_0}{\la_{max}}rM+\frac{2e^{\frac{M}{\la_{max}}}}{M}\e_0\\
	&\leq \left[\frac{rM}{\la_{max}}+\frac{2e^{\frac{M}{\la_{max}}}}{M}\right]\e_0.
\end{align*}
This contradicts with the assumption \eqref{condition-eps-1} on $\e_0$. Hence, we prove Proposition \ref{theorem:blow-up-C1-1}.
\end{proof}
 
\subsection{Proof of Theorem \ref{theorem:blow-up-C1-2}}\label{sec:thm-2}
Before we prove the blow up result of Theorem \ref{theorem:blow-up-C1-2} we construct an example of data $\U_0$ in dimension $2$ such that it satisfies \eqref{smallness-Hm}.
\begin{example}\label{example-1}
		We consider $\U_0:\R^2\rr\R^2$ defined as $\U_0(x_1,x_2)=(u_1,u_2)$
		\begin{align}
			u_1(x_1,x_2)&=-x_1\left(1+\frac{x_1^2}{R^2}\right)^{-n}\varphi\left(\left(1+\frac{x_1^2}{R^2}\right)\frac{x_2}{R}\right),\\ u_2(x_1,x_2)&=-x_2\left(1+\frac{x_2^2}{R^2}\right)^{-n}\varphi\left(\left(1+\frac{x_2^2}{R^2}\right)\frac{x_1}{R}\right),
		\end{align}
	 for some $\varphi\in C^\f_c(\R),R>0,n\in\mathbb{N}$ such that 
		\begin{equation}
			\left.\begin{array}{rl}
				\varphi(y)=1&\mbox{ if }\abs{y}\leq 1,\\
				0\leq \varphi(y)\leq 1&\mbox{ if }1\leq \abs{y}\leq 2,\\
					\varphi(y)=0 &\mbox{ if }\abs{y}\geq2.
			\end{array}\right\}
		\end{equation}
	We calculate for $j=k+l\geq1$,
	\begin{align}
		\pa_{x_1}^k\pa_{x_2}^lu_1=\frac{1}{R^{j-1}}\sum\limits_{i=0}^{j}\mathcal{P}_{i}^{k,l}\left(\frac{x_1}{R},\frac{x_2}{R}\right)\left(1+\frac{x_1^2}{R^2}\right)^{-n-i}\varphi^{(j-i)}\left(\left(1+\frac{x_1^2}{R^2}\right)\frac{x_2}{R}\right)
	\end{align}
where $\mathcal{P}^{k,l}_i$ is a polynomial of two variables. Subsequently, we obtain
\begin{equation*}
	\norm{\nabla_x^ju_1}_{L^2(\R^2)}=\mathcal{O}(1)R^{-j+\frac{3}{2}}\mbox{ for }1\leq j\leq 5.
\end{equation*}
Hence, we have $\norm{\nabla^2_xu_1}_{H^3(\R^2)}=\mathcal{O}(1)R^{-\frac{1}{2}}$. By a similar argument we can also have $\norm{\nabla^2_xu_2}_{H^3(\R^2)}=\mathcal{O}(1)R^{-\frac{1}{2}}$. We also note that
\begin{equation}
	\nabla_x\U(0,0)=\begin{pmatrix}
		\pa_{x_1}u_1(0,0)& 	\pa_{x_2}u_1(0,0)\\
			\pa_{x_1}u_2(0,0)& 	\pa_{x_2}u_2(0,0)
	\end{pmatrix}=\begin{pmatrix}
	-1&0\\
	0&-1
\end{pmatrix}.
\end{equation}
Choosing large enough $R>1$, we can see that $\U_0$ satisfies \eqref{smallness-Hm} and \eqref{condition-neg-1} with $\la_{max}=1$.
\end{example}
Next, we construct another example of data such that it satisfies \eqref{smallness-Hm} but it fails to satisfy the integral condition \eqref{condition-Sideris}. 
\begin{example}\label{example-2}
	Let us consider $\U_0$ defined as in Example \ref{example-1}. We now define $\tilde{\U}_0$ as $\psi_\la\U_0$ where $\psi\in C_c^\f$ having the property
	\begin{equation}
		\left.
		\begin{array}{ll}
			\psi_\la(x)=1&\mbox{ for }x\in B(0,2R),\\
			0\leq \psi_\la(x)\leq 1&\mbox{ for }2R\leq \abs{x}\leq \la,\\
			\psi_\la(x)=0&\mbox{ for }\abs{x}\geq \la,
		\end{array}\right\}
	\end{equation}
for some $\la>2R$ where $R$ is as Example \ref{example-1}. We set $\varrho_0(x)=\overline{\varrho}$ for all $x\in\R^2$ and some constant $\overline{\varrho}$. We choose $\overline{\varrho}$ large enough so that we have
\begin{equation}
	\frac{1}{\omega_2\la^{3}}\int\limits_{\R^2}\varrho_0(x)\U_0(x)\cdot x\,dx\leq \overline{\varrho}\frac{C_1R^4}{\omega_2\la^3}\mbox{ for some }C_1>0.
\end{equation}
Choosing $\la$ large enough and $\overline{\varrho}\geq1$, we can have $\overline{\varrho}\frac{C_1R^4}{\omega_2\la^3}\leq 3\sqrt{\gamma}(\overline{\varrho})^\frac{\gamma-1}{2}(\overline{\varrho})$. This shows that $(\varrho_0,\tilde{\U}_0)$ does not satisfy the integral condition \eqref{condition-Sideris}. 
\end{example}
In the following example, we consider a particular radially symmetric initial velocity $\U_0$ such that it satisfies smallness condition \eqref{smallness-Hm} and having a decreasing direction as in \eqref{condition-neg-1}.
\begin{example}[Radially symmetric data]
	We consider initial velocity $\U_0:\R^2\rr\R^2$ defined as follows
	\begin{equation*}
		\U_0(x)=h(r)\frac{x}{\abs{x}}\mbox{ where }r=\abs{x}.
	\end{equation*}
By an elementary calculation we get
\begin{equation*}
	\nabla_x\U_0(x)=\frac{h^\p(r)}{r^2}x\otimes x+\frac{h(r)}{r^3}\begin{pmatrix}
		x_2^2&x_1x_2\\
		x_1x_2&x_1^2
	\end{pmatrix}.
\end{equation*}
We write $a_{ij}=\pa_iu_j(x)=\frac{h^\p(r)}{r^2}\mathcal{P}^{1}_2(x_1,x_2)+\frac{h(r)}{r^3}\mathcal{P}^{2}_2(x_1,x_2)$ where $\mathcal{P}^{j}_2(x_1,x_2),j=1,2$ are polynomials of degree 2. Then we have
\begin{align*}
	\pa_ka_{ij}&=\frac{h^{(2)}(r)}{r^3}\mathcal{P}^{1}_3(x_1,x_2)+\frac{h^{(1)}(r)}{r^4}\mathcal{P}^{2}_3(x_1,x_2)+\frac{h(r)}{r^5}\mathcal{P}^{3}_3(x_1,x_2),\\
	\pa_l(\pa_ka_{ij})&=\frac{h^{(3)}(r)}{r^4}\mathcal{P}^{1}_4(x_1,x_2)+\frac{h^{(2)}(r)}{r^5}\mathcal{P}^{2}_4(x_1,x_2)+\frac{h^{(1)}(r)}{r^6}\mathcal{P}^{3}_4(x_1,x_2)+\frac{h(r)}{r^7}\mathcal{P}^{4}_4(x_1,x_2),\\
		\pa^2_{lm}(\pa_ka_{ij})&=\frac{h^{(4)}(r)}{r^5}\mathcal{P}^{1}_5(x_1,x_2)+\frac{h^{(3)}(r)}{r^6}\mathcal{P}^{2}_5(x_1,x_2)+\frac{h^{(2)}(r)}{r^7}\mathcal{P}^{3}_5(x_1,x_2)+\frac{h^{(1)}(r)}{r^8}\mathcal{P}^{4}_5(x_1,x_2)\\
		&+\frac{h(r)}{r^9}\mathcal{P}^{5}_5(x_1,x_2),\\
	\pa^3_{lmn}(\pa_ka_{ij})&=\frac{h^{(5)}(r)}{r^6}\mathcal{P}^{1}_6(x_1,x_2)+\frac{h^{(4)}(r)}{r^7}\mathcal{P}^{2}_6(x_1,x_2)+\frac{h^{(3)}(r)}{r^8}\mathcal{P}^{3}_6(x_1,x_2)+\frac{h^{(2)}(r)}{r^9}\mathcal{P}^{4}_6(x_1,x_2)\\
		&+\frac{h^{(1)}(r)}{r^{10}}\mathcal{P}^{5}_6(x_1,x_2)+\frac{h(r)}{r^{11}}\mathcal{P}^{6}_6(x_1,x_2).
\end{align*}
Now, we choose $h(z)=Re^{-\frac{z}{R}}\varphi(z-R)$ for $R>1$ and $\varphi$ is defined as
\begin{equation*}
		\left.
	\begin{array}{ll}
		\varphi(z)=1&\mbox{ for }\abs{z}\leq1,\\
		0\leq	\varphi(z)\leq 1&\mbox{ for }1\leq\abs{z}\leq 2,\\
			\varphi(z)=0&\mbox{ for } \abs{z}\geq 2.
	\end{array}\right\}
\end{equation*}
Then we have
\begin{equation*}
	\norm{	\nabla_x^l(\pa_ka_{ij})}_{L^2(\R^2)}=\mathcal{O}(1) R^{-1-l}\mbox{ for }l=0,1,2,3.
\end{equation*}
Further, we note that $	\nabla_x\U_0(R,0)=\begin{pmatrix} -1&0\\ 0&1\end{pmatrix}$. Hence, the initial velocity satisfies \eqref{condition-neg-1} with $\la_{max}=1$ and we can choose an appropriate $\varrho_0$ and $R>1$ such that \eqref{smallness-Hm} is satisfied.
\end{example}
See \cite{Li-Wang,MRRS-II} for studies on blow up of spherically symmetric solutions.

Now, we prove the blow up result for $C^1$ solution when initial data satisfies \eqref{condition-neg-1} and \eqref{smallness-Hm}. The proof is based on Kato type energy estimate derived as in section \ref{sec:energy-estimate} and deduction of a Riccati equation.
\begin{proof}[Proof of Theorem \ref{theorem:blow-up-C1-2}:]
	 We argue by contradiction. Suppose that $(\varrho,\U)\in C^1([0,T]\times\R^d;[0,\f)\times\R^d))$ for some $T>2/\la_{max}$. Similar to Proposition \ref{theorem:blow-up-C1-1}, we consider the following symmetrization of \eqref{eqn-isen-1}--\eqref{eqn-isen-2} in the spirit of \cite{KMU},
	 \begin{align}
	 	\pa_t\pi+\U\cdot\nabla_x\pi+C_1\pi\dv_x\U&=0,\label{blow-2-sym-1}\\
	 	\pa_t\U+\U\cdot\nabla_x\U+C_1\pi\nabla_x\pi&=0,\label{blow-2-sym-2}
	 \end{align}	
	 where $C_1=\frac{\gamma-1}{2}$ and $\pi=\sqrt{\frac{\gamma-1}{4\gamma}}\varrho^\frac{\gamma-1}{2}$.  Let us consider the flow $X(t,x_0)$ defined as follows,
	\begin{equation*}
		\frac{dX(t,x_0)}{dt}=\U(t,X(t,x_0))\mbox{ and }X(0,x_0)=x_0.
	\end{equation*}
	By using energy estimate \eqref{energy-est-1} we obtain
	\begin{equation}
	 \sum\limits_{k=0}^{m}\int\limits_{\R^d}\left[\abs{\nabla_x^{k+2}\pi(t,\cdot)}^2+\abs{\nabla_x^{k+2}\U(t,\cdot)}^2\right]\,dx\leq e^{C_2t}\sum\limits_{k=0}^{m}\int\limits_{\R^d}\left[\abs{\nabla_x^{k+2}\pi_0}^2+\abs{\nabla_x^{k+2}\U_0}^2\right]\,dx.
	\end{equation}
   Then from assumption \eqref{smallness-Hm} and Lemma \ref{lemma:Morrey} we have the smallness of $\norm{\nabla^2_x\pi(t,\cdot)}_{L^\f(\R^d)}$ as 
   \begin{equation}\label{estimate-pi-thm2}
   	\norm{\nabla^2_x\pi(t,\cdot)}_{L^\f(\R^d)}\leq e^{C_2t}\kappa\mbox{ where }\kappa=\sum\limits_{k=0}^{m}\int\limits_{\R^d}\left[\abs{\nabla_x^{k+2}\pi_0}^2+\abs{\nabla_x^{k+2}\U_0}^2\right]\,dx.
   \end{equation}
Subsequently, from \eqref{blow-2-sym-1}, we obtain
\begin{equation}
	0\leq \pi(t,X(t,x_0))\leq e^{C_2 t}\pi(x_0).
\end{equation}
     From \eqref{condition-neg-1}, there exists $x_0\in\R^d$ and $\xi_0\in\mathbb{S}^{d-1}$ such that
   \begin{equation*}
   	\sup\limits_{\xi\in\mathbb{S}^{d-1}}\left[-\nabla_x\U_0(x_0):\xi\otimes\xi\right]=-\nabla_x\U_0(x_0):\xi_0\otimes\xi_0= \la_{max}.
   \end{equation*}
  Applying $\nabla_x$ on both side of \eqref{blow-2-sym-2}, we have
   \begin{equation}
  	\pa_t\nabla\U+\U\cdot\nabla^2_x\U+\nabla_x\U\nabla_x\U+C_1\pi\nabla^2_x\pi+C_1\nabla_x\pi\otimes\nabla_x\pi=0.
  \end{equation}
	First we observe that
  \begin{align}
  	\pa_t[\nabla_x\U-\nabla_x^T\U]+\U\cdot\nabla_x[\nabla_x\U-\nabla_x^T\U]&=\nabla_x^T\U\nabla_x^T\U-\nabla_x\U\nabla_x\U\\
  	&=\nabla_x^T\U[\nabla_x^T\U-\nabla_x\U]+[\nabla^T_x\U-\nabla_x\U]\nabla_x\U.
  \end{align}
Along the curve $X(t,x_0)$ we have $\frac{d}{dt}\abs{[\nabla_x\U-\nabla_x^T\U](t,X(t))}\leq C_1\abs{[\nabla_x\U-\nabla_x^T\U](t,X(t))}$. By Gr\"onwall's lemma we conclude that $\nabla_x\U(t,X(t))=-\nabla_x^T\U(t,X(t))$ since $\nabla_x\U_0(x_0)=-\nabla_x^T\U_0(x_0)$. Now, for any $\xi\in\mathbb{S}^{d-1}$ we have,
  \begin{align}
  		\pa_t[-\nabla\U:\xi\otimes\xi]+\U\cdot\nabla_x[-\nabla\U:\xi\otimes\xi]
  		&=[\nabla_x\U\nabla_x\U]:\xi\otimes\xi+C_1\pi\nabla^2_x\pi:\xi\otimes\xi+C_1\abs{\xi\cdot\nabla_x\pi}^2\nonumber\\
  		&\geq [\nabla_x\U\nabla_x\U]:\xi\otimes\xi+C_1\pi\nabla^2_x\pi:\xi\otimes\xi.\label{eqn-derivative-u}
  \end{align}
We note that $[\nabla_x\U_0(x_0)\nabla_x\U_0(x_0)]:\xi_0\otimes\xi_0=\la_{max}^2$ since $-\nabla\U_0(x_0)\xi_0=\la_{max}\xi_0$. Let $\de>0$. Since $\U\in C^1([0,T)\times \R^d)$ there exists $t_1>0$ such that the following holds
\begin{equation}\label{estimate-derivative-U}
	[\nabla_x\U(t,X(t,x_0))\nabla_x\U(t,X(t,x_0))]:\xi_0\otimes\xi_0\geq \abs{\nabla_x\U(t,X(t,x_0)):\xi_0\otimes\xi_0}^2-\de^2.
\end{equation}
Using \eqref{estimate-pi-thm2}, \eqref{eqn-derivative-u} and \eqref{estimate-derivative-U} we get
   \begin{equation}
	\frac{d}{dt}[-\nabla\U(t,X(t,x_0)):\xi_0\otimes\xi_0] \geq \abs{\nabla_x\U(t,X(t,x_0)):\xi_0\otimes\xi_0}^2-\de^2-C_1\kappa\pi(x_0)e^{2C_2t}.
   \end{equation}
Define $a(t):=[-\nabla\U(t,X(t,x_0)):\xi_0\otimes\xi_0] $. For sufficiently small $\de>0$ we have
\begin{equation}
	a^\p(t)\geq a^2(t)-2C_1\kappa\pi(x_0)e^{2C_2t}.
\end{equation}
From the assumption \eqref{smallness-Hm}, we have $\kappa \pi_0(x_0)\leq \e_0$. Hence,
\begin{equation}
	a^\p(t)\geq a^2(t)-2C_1\e_0e^{2C_2t}.
\end{equation}
Next, we consider $b(t)=a^m(t)$. Then we have
\begin{align*}
	b^\p(t)=ma^{m-1}(t)a^\p(t)&\geq ma^{m-1}(t)[a^2(t)-2C_1\e_0e^{2C_2t}]\\
	&=ma^{m+1}(t)-2C_1m\e_0e^{2C_2t} a^{m-1}\\
	&=mb^{1+\frac{1}{m}}(t)-2C_1m\e_0e^{2C_2t}b^{1-\frac{1}{m}}(t).
\end{align*}
We define $\La(t)=b(t)e^{-At}$. Then we have
\begin{align*}
	\La^\p(t)&=b^\p(t)e^{-At}-Ab(t)e^{-At}\\
	&\geq mb^{1+\frac{1}{m}}(t)e^{-At}-Ab(t)e^{-At}-2C_1m\e_0e^{2C_2t}b^{1-\frac{1}{m}}(t)e^{-At}\\
	&=mb^{1+\frac{1}{m}}(t)e^{-At-\frac{At}{m}}e^{\frac{At}{m}}-A\La(t)-2C_1m\e_0e^{2C_2t}b^{1-\frac{1}{m}}(t)e^{-At}\\ 
	&=m\La^{1+\frac{1}{m}}(t)e^{\frac{At}{m}}-A\La(t)-2C_1m\e_0e^{2C_2t}b^{1-\frac{1}{m}}(t)e^{-At}.
\end{align*}
We choose large $A$ such that $A\geq 2C_2$. Then we have
\begin{equation*}
	\La^\p(t)\geq m\La^{1+\frac{1}{m}}(t)e^{\frac{At}{m}}-A\La(t)-2C_1m\e_0b^{1-\frac{1}{m}}(t).
\end{equation*}
Using $b(t)=\La(t) e^{At}$, we get
\begin{align*}
	\La^\p(t)&\geq m\La^{1+\frac{1}{m}}(t)e^{\frac{At}{m}}-A\La(t)-2C_1m\e_0\La^{1-\frac{1}{m}}(t)e^{At-\frac{At}{m}}\\
	&=m\La^{1+\frac{1}{m}}(t)e^{\frac{At}{m}}\left[1-2C_1\e_0\La^{-\frac{2}{m}}(t)e^{-\frac{At}{m}}\right]-A\La(t).
\end{align*}
We note that $\La^\frac{2}{m}(t)=b^\frac{2}{m}(t)e^{-\frac{2At}{m}}=a^2(t)e^{-\frac{2At}{m}}$. We choose large enough $m>1$ such that $e^{\frac{At}{m}}\leq \frac{5}{4}$ for all $t\in[0,T]$. From the assumption on initial data, we have
\begin{equation*}
	2C_1\e_0\La^{-\frac{2}{m}}(t)e^{-\frac{At}{m}}=2C_1\e_0a^{-2}(t)e^{\frac{At}{m}}\leq \frac{5C_1}{2}\e_0a^{-2}(t).
\end{equation*}
Since $a^\p(t)\geq0$, we have $a(t)\geq a(0)=\la_{max}-\de$, from the choice of $\e_0$ (as in Theorem \ref{theorem:blow-up-C1-2}) we get
\begin{equation*}
		2C_1\e_0 \La^{-\frac{2}{m}}(t)e^{-\frac{At}{m}}\leq \frac{5C_1}{2}\e_0 (\la_{max}-\de)^{-2}\leq \frac{1}{4}\mbox{ for small enough }\de>0.
\end{equation*}
Therefore, we have
\begin{align*}
	\La^\p(t)&\geq \frac{3m}{4}\La^{1+\frac{1}{m}}(t)-A\La(t)\\
	             &\geq \frac{3m}{4}\La^{1+\frac{1}{m}}(t)\left[1-\frac{4A}{3m}\La^{-\frac{1}{m}}(t)\right].
\end{align*}
For large enough $m>1$, we have $\frac{A}{m}\La^{-\frac{1}{m}}(t)\leq \frac{1}{4}$. Hence,
\begin{equation}
	\La^\p(t)\geq\frac{m}{2}\La^{1+\frac{1}{m}}(t).  
\end{equation}
Therefore we have
\begin{equation}
	m\left[\frac{1}{\La^\frac{1}{m}(0)}-\frac{1}{\La^\frac{1}{m}(t)}\right]\geq \frac{mt}{2},
\end{equation}
or equivalently, 
\begin{equation}
	\frac{1}{\La^\frac{1}{m}(t)}\leq \frac{1}{\La^\frac{1}{m}(0)}-\frac{t}{2}.  
\end{equation}
Recall that $\La^\frac{1}{m}(t)=(b(t)e^{-At})^\frac{1}{m}=b^\frac{1}{m}(t)e^{-\frac{At}{m}}=a(t)e^{-\frac{At}{m}}$ and $\La^\frac{1}{m}(0)=b^\frac{1}{m}(0)=a(0)$. Therefore, 
\begin{equation}\label{eq-m-a(t)}
	\frac{1}{a(t)e^{-\frac{At}{m}}}\leq \frac{1}{a(0)}-\frac{t}{2}\mbox{ for }t\in[0,t_1].
\end{equation}
Since \eqref{eq-m-a(t)} remains true for all $m$ sufficiently large, we obtain
\begin{equation}
		\frac{1}{a(t)}\leq \frac{1}{a(0)}-\frac{t}{2}\mbox{ for all }t\in[0,t_1].
\end{equation}
Therefore,
\begin{equation}
	\frac{1}{a(t_1)}\leq \frac{1}{a(0)}-\frac{t_1}{2}.
\end{equation}
Let $\nu(t)$ be defined as $\nu(t):=\sup\limits_{\xi\in\mathbb{S}^{d-1}}[-\nabla\U(t,X(t,x_0)):\xi\otimes\xi]$. We observe that $\nu(0)=a(0)$ and $\nu(t_1)\geq a(t_1)$. Therefore, 
\begin{equation}
	\frac{1}{\nu(t_1)}\leq \frac{1}{a(t_1)}\leq \frac{1}{a(0)}-\frac{t_1}{2}=\frac{1}{\nu(0)}-\frac{t_1}{2}.
\end{equation} Now, we apply the same argument at time $t_1$ and $x_1=X(t_1,x_0)$. In this case, we reuse the previous method for direction $\xi_1\in\mathbb{S}^{d-1}$ satisfying $\nu(t_1)=[-\nabla\U(t,X(t,x_0)):\xi_1\otimes\xi_1]$. We note that $\nu(t_1)\geq \nu(0)$. Hence, there exists $t_2\geq 2t_1$ such that \eqref{estimate-derivative-U} holds. We repeat the same argument for time $t_1,t_2$ we can obtain
\begin{equation}
		\frac{1}{\nu(t_2)}\leq \frac{1}{\nu(t_1)}-\frac{t_2-t_1}{2}.
\end{equation}
Since $\U,\pi$ belong to $C^1([0,T]\times\R^d)$, for any $t^*\in(0,T)$, we have finitely many points $0=t_0<t_1<t_2<\cdots<t_k=t^*$ such that
\begin{equation}
	\frac{1}{\nu(t_{j})}\leq \frac{1}{\nu(t_{j-1})}-\frac{t_{j+1}-t_j}{2}\mbox{ for all }1\leq j\leq k.
\end{equation}
Hence, we get
\begin{equation}
	\frac{1}{\nu(t_*)}\leq \frac{1}{\nu(0)}-\sum\limits_{j=1}^{k}\frac{t_{j}-t_{j-1}}{2}=\frac{1}{\nu(0)}-\frac{t^*}{2},
\end{equation}
or equivalently,
\begin{equation}
	\nu(t^*)\geq \frac{2\nu(0)}{2-\nu(0)t^*}.
\end{equation}
Hence, $\nu(t)\rr\f$ as $t\rr 2/\nu(0)$.  This completes the proof of Theorem \ref{theorem:blow-up-C1-2}. 

\end{proof}
\section{Break down of continuity of a solution}\label{sec:break-down-cont}

Adapting \cite[Theorem 4.1.1]{Dafermos}, we have the following result.
\begin{lemma}\label{lemma:fsp}
	Let $(\varrho,\U)\in C([0,T^*)\times\R^d)$ be an admissible solution to \eqref{eqn-isen-1}--\eqref{eqn-isen-2} (as in Definition \ref{defn:admissible}) corresponding to initial data $(\varrho_0,\U_0)$. Suppose $(\varrho_0,\U_0)$ satisfies
	\begin{equation}
		(\varrho_0(x),\U_0(x))=(\overline{\varrho},\textbf{0})\mbox{ for }\abs{x}\geq R\mbox{ for some }R>0.
	\end{equation}
Then we have
\begin{equation}
	(\varrho(t,\cdot),\U(t,\cdot))=(\overline{\varrho},\textbf{0})\mbox{ for }\abs{x}\geq R+\si t\mbox{ for }t\in[0,T^*)\mbox{ and }\si=\sqrt{\gamma}(\overline{\varrho})^{\frac{\gamma-1}{2}}.
\end{equation}
\end{lemma}
Though in \cite[Theorem 4.1.1]{Dafermos}, the result has been proved for classical $C^1$ solution but we can repeat the proof with the assumption of entropy inequality \eqref{ineq:entropy}. 
\begin{remark}\label{remark-blow-up-cont}
	\begin{enumerate}
		\item We remark that if a weak solutions having H\"older regularity $C^{\frac{1}{3}}$ or if it is continuous and BV the equality holds in \eqref{ineq:entropy}. We refer to \cite{FGGW} for details in this regards.
		\item We also remark that the result of finite speed of propagation remains valid even for $L^\f$ solutions but in that case we must take $\si$ as follows
		\begin{equation}\label{speed:L-infty}
			\si=\sup\limits_{\xi\in\mathbb{S}^{d-1},(t,x)\in[0,T^*)\times\R^d}\abs{\U(t,x)\cdot\xi\pm\sqrt{\ga}( \varrho(t,x))^{\frac{\ga-1}{2}}}.
		\end{equation}
	   \item We also note that the proof of Theorem \ref{theorem-cont} can also be replicated for $L^\f$ solutions but $L^\f$--data $\varrho_0,\U_0$ fails to satisfy the condition \eqref{condition-Sideris}. Indeed, we have $\abs{\U_0}\leq \si$ where $\si$ is defined as in \eqref{speed:L-infty}. Subsequently,
	   \begin{equation*}
	   	\frac{1}{\omega_dR^{d+1}}\int\limits_{B(\textbf{0},R)}\varrho_0\U_0\cdot x\,dx\leq \frac{\norm{\varrho_0}_{L^\f(\R^d)}\si}{\omega_dR^{d+1}}\int\limits_{B(\textbf{0},R)}\abs{x}\,dx= \frac{\norm{\varrho_0}_{L^\f(\R^d)}\si}{d+1}.
	   \end{equation*}
	\end{enumerate}
\end{remark}
\begin{proof}[Proof of Lemma \ref{lemma:fsp}:]
	 Consider $\eta(\varrho,\textbf{m})$ and $Q(\varrho,\textbf{m})$ as follows
	\begin{align}
		\eta(\varrho,\textbf{m})&=\frac{\abs{\textbf{m}}^2}{2\varrho}+P(\varrho)-P^\p(\overline{\varrho})(\varrho-\overline{\varrho})-P(\overline{\varrho}),\\
		Q(\varrho,\textbf{m})&=\frac{\textbf{m}}{\varrho}\left[\frac{\abs{\textbf{m}}^2}{2\varrho}+P(\varrho)-P^\p(\overline{\varrho})(\varrho-\overline{\varrho})-P(\overline{\varrho})\right].
	\end{align}
For $s\in \R$ and $\xi\in \mathbb{S}^{d-1}$ we define $\Phi(s,\xi;\varrho,\textbf{m})=s\eta(\varrho,\textbf{m})-\xi\cdot Q(\varrho,\textbf{m})$. We also set 
\begin{equation*}
	F_\xi(\varrho,\textbf{m})=\left(\textbf{m}\cdot\xi,\frac{\textbf{m}\cdot\xi}{\varrho^2}\textbf{m}+p(\varrho)\xi\right).
\end{equation*} 
We note that
\begin{align}
	&\Phi(s,\xi;\overline{\varrho},\textbf{0})=0,\quad \nabla_{\varrho,\textbf{m}}\Phi(s,\xi;\overline{\varrho},\textbf{0})=\textbf{0},\\
	&\nabla_{\varrho,\textbf{m}}^2\Phi(s,\xi;\overline{\varrho},\textbf{0})=\nabla^2\eta(\overline{\varrho},\textbf{0})\left[s\mathbb{I}_{d+1}-\nabla_{\varrho,\textbf{m}}F_\xi(\overline{\varrho},\textbf{0})\right].
\end{align}
Let eigenvalues of $F_\xi(\overline{\varrho},\textbf{0})$ be $\la_1(\overline{\varrho},\textbf{0})=-\sqrt{p^\p(\overline{\varrho})},\la_{d+1}(\overline{\varrho},\textbf{0})=\sqrt{p^\p(\overline{\varrho})}$ and $\la_j(\overline{\varrho},\textbf{0})=0,\,2\leq j\leq d$. Let $\zeta_j,1\leq j\leq d+1$ be eigenvectors corresponding to $\la_j,1\leq j\leq d+1$ respectively. Subsequently, we have
\begin{equation}
	\nabla_{\varrho,\textbf{m}}^2\Phi(s,\xi;\overline{\varrho},\textbf{0}):\zeta_j\otimes \zeta_j=(s-\la_j(\overline{\varrho},\textbf{0}))\nabla^2\eta(\overline{\varrho},\textbf{0}):\zeta_j\otimes \zeta_j
\end{equation}
Hence, for $s=\sqrt{p^\p(\overline{\varrho})}+\e$ with $\e>0$ we get
\begin{equation*}
		\nabla_{\varrho,\textbf{m}}^2\Phi(\sqrt{p^\p(\overline{\varrho})}+\e,\xi;\overline{\varrho},\textbf{0}):\zeta_j\otimes \zeta_j\geq C_0\e\mbox{ for some }C_0>0.
\end{equation*}
From the smoothness of $(\varrho,\textbf{m})\mapsto \Phi(s,\xi;\varrho,\textbf{m})$ there exists $\de=\de(\e)>0$ such that
\begin{equation}\label{estimate-lower-Phi}
	\Phi(\sqrt{p^\p(\overline{\varrho})}+\e,\xi;\varrho,\textbf{m})\geq \frac{\e}{4}\left[\abs{\textbf{m}}^2+\abs{\varrho-\overline{\varrho}}^2\right]\mbox{ for }\abs{\textbf{m}}+\abs{\varrho-\overline{\varrho}}\leq \de\mbox{ for all }\xi\in\mathbb{S}^{d-1}.
\end{equation}
	\begin{claim}\label{claim-1}
			Let $(\varrho,\U)\in C([0,T^*)\times\R^d)$ be an admissible solution to \eqref{eqn-isen-1}--\eqref{eqn-isen-2} (as in Definition \ref{defn:admissible}) corresponding to initial data $(\varrho_0,\U_0)$. Suppose $(\varrho_0,\U_0)$ satisfies
		\begin{equation}
			(\varrho_0(x),\U_0(x))=(\overline{\varrho},\textbf{0})\mbox{ if }x\cdot\xi\geq R\mbox{ for some }\xi\in\mathbb{S}^{d-1}\mbox{ and }R>0.
		\end{equation}
		Then we have
		\begin{equation}
			(\varrho(t,\cdot),\U(t,\cdot))=(\overline{\varrho},\textbf{0})\mbox{ for }x\cdot\xi\geq R+\si t\mbox{ for }t\in[0,T^*)\mbox{ and }\si=\sqrt{\gamma}(\overline{\varrho})^{\frac{\gamma-1}{2}}.
		\end{equation}
	\end{claim}
\begin{proof}[Proof of Claim \ref{claim-1}:] We are going to prove this by contradiction. Fix $\e>0$. Suppose there exists a point $(t_0,x_0)\in(0,T^*)\times\R^d$ such that (i) $x_0\cdot\xi\geq R+\si t_0$ and (ii) $(\varrho(t_0,x_0),\U(t_0,x_0))\neq (\overline{\varrho},\textbf{0})$. Without loss of generality we can assume that $t_0<\de=\de(\e)$ if required we choose the minimal time $t^*$ as $t_0$ when (i) and (ii) holds and we replace $R$ by $R+\si \left(t^*-\frac{\de}{2}\right)$. Now consider the following set
	\begin{equation}
		\mathcal{J}_{t_0,x_0}:=\left\{(t,x)\in[0,\f)\times\R^d);x\cdot\xi\geq x_0\cdot\xi-\si_\e (t_0-t),\abs{x-x_0}\leq \si_\e(t-t_0),t\in[0,t_0]\right\},
	\end{equation}
	where $\si_\e=\si+\e$. We define $\mathcal{J}^n:=\left\{(t,x)\in\mathcal{J}_{t_0,x_0};\mbox{dist}((t,x),\pa\mathcal{J}_{t_0,x_0})\geq1/n\right\}$ where $\pa\mathcal{J}_{t_0,x_0}$ denotes the boundary of $\mathcal{J}_{t_0,x_0}$. Consider $\psi_n:[0,\f)\times\R^d\rr \R$ such that $\psi\in C^\f_c([0,\f)\times\R^d)$ satisfying 
	\begin{equation}
		\left.
		\begin{array}{ll}
			\psi_n(t,x)=1&\mbox{ if }(t,x)\in\mathcal{J}^n,\\
			0\leq 	\psi_n(t,x)\leq 1&\mbox{ if }(t,x)\in\mathcal{J}_{t_0,x_0}\setminus\mathcal{J}^n,\\
				\psi_n(t,x)=0&\mbox{ if }(t,x)\in\mathcal{J}^c_{t_0,x_0},
		\end{array}\right\}
	\end{equation} 
where $A^c$ denotes the complement of set $A$ in whole space. Taking $\varphi=\psi_n$ in \eqref{ineq:entropy} we obtain 
\begin{equation}
	\int\limits_{\pa\mathcal{J}_{t_0,x_0}}\Phi(s(t,y),\nu(t,y);\varrho(t,y),\textbf{m}(t,y))d\mathcal{S}_{t,y}\leq0,
\end{equation}
where $d\mathcal{S}_{t,y}$ is the surface measure on $\pa\mathcal{J}_{t_0,x_0}$ and $(s(t,y),\nu(t,y))$ is the outward normal at $(t,y)$. We observe that $\pa\mathcal{J}_{t_0,x_0}=\mathcal{B}_1\cup\mathcal{B}_2\cup\mathcal{B}_3$ where
\begin{align*}
	\mathcal{B}_1&=\{(0,x);x\cdot\xi\geq x_0\cdot\xi-\si_\e t_0,\abs{x-x_0}\leq \si_\e t_0\},\\
	\mathcal{B}_2&=\{(0,x);x\cdot\xi = x_0\cdot\xi-\si_\e (t_0-t)\},\\
	\mathcal{B}_3&=\{(0,x);x\cdot\xi \geq x_0\cdot\xi-\si_\e (t_0-t),\abs{x-x_0}= \si_\e (t_0-t)\}.
\end{align*}
Therefore, we have
\begin{equation*}
	(s(t,y),\nu(t,y))=\left\{
	\begin{array}{rl}
	(-1,0)&\mbox{ if }(t,y)\in\mathcal{B}_1,\\
	(\si_\e,-\xi)&\mbox{ if }(t,y)\in\mathcal{B}_2,\\
	\left(\si_\e,\frac{y-x_0}{\si_\e(t_0-t)}\right)	&\mbox{ if }(t,y)\in\mathcal{B}_3.
	\end{array}\right.
\end{equation*}
We also observe that $(t_0,x_0)$ is chosen such a way that $\abs{\varrho(t,x)-\overline{\varrho}}+\abs{\textbf{m}(t,x)}\leq \de$ for all $(t,x)\in \mathcal{J}_{t_0,x_0}$. Hence, by using \eqref{estimate-lower-Phi} we get a contradiction. This concludes the proof of Claim \ref{claim-1}.
\end{proof}
To complete the proof of Lemma \ref{lemma:fsp} let us consider a point $x_0\in\R^d$ such that $\abs{x_0}=R$. Set $\xi=x_0/\abs{x_0}$. By Claim \ref{claim-1} we have $(\varrho(t,x),\textbf{m}(t,x))=(\overline{\varrho},\textbf{0})$ for $x\cdot \xi\geq x_0\cdot\xi+\si t$. Note that $x_0\cdot\xi=\abs{x_0}=R$. We further observe that
\begin{equation}
\{(t,x)\in[0,\f)\times\R^d;\abs{x-x_0}\geq R+\si t\}=\bigcup\limits_{\xi\in\mathbb{S}^{d-1}}\{(t,x)\in[0,\f)\times\R^d;x\cdot\xi\geq R+\si t\}
\end{equation}
This completes the proof of Lemma \ref{lemma:fsp}.
\end{proof}
\subsection{Proof of Theorem \ref{theorem-cont}}
\begin{proof}[Proof of Theorem \ref{theorem-cont}:]
	Let us consider $F(t)$ and $M(t)$ defined as follows
	\begin{equation*}
		M(t):=\int\limits_{\R^d}(\varrho(t,x)-\bar{\varrho})\,dx\mbox{ and }F(t):=\int\limits_{\R^d}x\cdot \varrho\U\,dx.
	\end{equation*}
	Let $\{\eta_\e\}_{\e>0}$ be Friedrich mollifiers. Then, we consider $F_\e,M_\e(t)$ defined as follows
	\begin{align}
		F_\e(t)=\int\limits_{\R^d}x\cdot( \varrho\U)_\e(t,\cdot)\,dx\mbox{ and }M_\e(t)=\int\limits_{\R^d}(\varrho_\e(t,x)-\bar{\varrho})\,dx
	\end{align}
	Then by Fundamental Theorem of Calculus we have
	\begin{align}
		F_\e(t_2)-F_\e(t_1)=\int\limits_{t_1}^{t_2}\int\limits_{\R^d}x\cdot \pa_t( \varrho\U)_\e(t,\cdot)\,dxdt\mbox{ for }t_2>t_1>2\e.
	\end{align}
	We can mollify the system \eqref{eqn-isen-1}--\eqref{eqn-isen-2} with $\eta_\e$ to obtain
	\begin{align}
		\pa_t\varrho_\e+\dv_x(\varrho\U)_\e&=0\label{eq-1}\\
		\pa_t(\varrho\U)_\e+\dv_x\left(\varrho\U\otimes\U\right)_\e+\nabla_xp(\varrho)_\e&=0.\label{eq-2}
	\end{align}
	By using \eqref{eq-2} we get
	\begin{align}
		F_\e(t_2)-F_\e(t_1)&=\int\limits_{t_1}^{t_2}\int\limits_{\R^N}x\cdot \pa_t( \varrho\U)_\e(t,\cdot)\,dxdt\nonumber\\
		&=-\int\limits_{t_1}^{t_2}\int\limits_{\R^d}x\cdot\left[\dv_x\left(\varrho\U\otimes\U\right)_\e+\nabla_xp(\varrho)_\e\right]\,dx\nonumber\\
		&=\int\limits_{t_1}^{t_2}\int\limits_{\R^d}\left[\mathbb{I}_d:\left(\varrho\U\otimes\U\right)_\e+d(p(\varrho)_\e-p(\bar{\varrho}))\right]\,dx\nonumber\\
		&=\int\limits_{t_1}^{t_2}\int\limits_{\R^d}\left[\left(\varrho\abs{\U}^2\right)_\e+d(p(\varrho)_\e-p(\bar{\varrho}))\right]\,dx.	\label{eqn-F-mollified}
	\end{align}
	Similarly, with the help of \eqref{eq-1} we obtain
	\begin{equation}
		M_\e(t_2)-M_\e(t_1)=\int\limits_{t_1}^{t_2}\int\limits_{\R^d}\pa_t(\varrho_\e-\bar{\varrho})\,dxdt=-\int\limits_{t_1}^{t_2}\int\limits_{\R^d}\dv_x(\varrho\U)_\e\,dxdt=0.\label{eqn-M-mollified}
	\end{equation}
	Passing to the limit in \eqref{eqn-F-mollified}, \eqref{eqn-M-mollified} as $\e\rr0$, we obtain
	\begin{equation}
		F(t_2)-F(t_1)=\int\limits_{t_1}^{t_2}\int\limits_{\R^d}\left[\left(\varrho\abs{\U}^2\right)+d(p(\varrho)-p(\bar{\varrho}))\right]\,dxdt\mbox{ and }M(t_2)=M(t_1)\mbox{ for }t_2>t_1\geq0
	\end{equation}
	Subsequently, we have
	\begin{equation*}
		\int\limits_{B(t)}(\varrho(t,\cdot)-\bar{\varrho})\,dx=\int\limits_{\R^d}(\varrho(t,\cdot)-\bar{\varrho})\,dx
		=M(t)=M(0)
		=\int\limits_{\R^d}(\varrho_0(\cdot)-\bar{\varrho})\,dx.
	\end{equation*}
	Note that
	\begin{equation}
		\int\limits_{B(t)}\varrho(t,\cdot)\,dx=M(t)+\bar{\varrho}\abs{B(t)}=M(0)+\bar{\varrho}\abs{B(t)}.
	\end{equation}
     Then, by using Jensen's inequality we obtain 
	\begin{align}
		\int\limits_{B(t)}p(\varrho)\,dx=\kappa\int\limits_{B(t)}\varrho^\gamma(t,\cdot)\,dx
		&\geq \kappa\abs{B(t)}^{1-\gamma}\left(\int\limits_{B(t)}\varrho(t,\cdot)\,dx\right)^{\gamma}\nonumber\\
		&= \kappa\abs{B(t)}^{1-\gamma}\left(M(0)+\abs{B(t)}\bar{\varrho}\right)^{\gamma}\nonumber\\
		&\geq  \kappa\abs{B(t)}^{1-\gamma}\left(\abs{B(t)}\bar{\varrho}\right)^{\gamma}\nonumber\\
		&=  \int\limits_{B(t)}p(\bar{\varrho})\,dx.\label{inequality-pressure}
	\end{align}
	Since, $\varrho, \U\in C([0,T^*)\times\R^d)$, we can pass $t_1\rr0$ and put $t_2=t$ to obtain
	\begin{equation}
		F(t)=F(0)+\int\limits_{0}^{t}\int\limits_{\R^d}\left[\left(\varrho\abs{\U}^2\right)+d(p(\varrho)-p(\bar{\varrho}))\right]\,dxds.
	\end{equation}
	Note that supports of $\U$ and $p(\varrho)-p(\bar{\varrho})$ is contained in $B(t)$. Hence, the following function, $\Psi$  is well-defined, 
	\begin{equation}
		\Psi(s):= \int\limits_{\R^d}\left[\left(\varrho\abs{\U}^2\right)+d(p(\varrho)-p(\bar{\varrho}))\right](s,\cdot)\,dx.
	\end{equation}
     By our assumption $\Psi$ is continuous as well. Therefore, we have $F\in C^1(0,T^*)$ and 
	\begin{equation}
		F^\p(t)=\int\limits_{\R^d}\left[\left(\varrho\abs{\U}^2\right)+d(p(\varrho)-p(\bar{\varrho}))\right](t,\cdot)\,dx\mbox{ for }t\in(0,T^*).
	\end{equation}
	From \eqref{inequality-pressure} we further obtain
	\begin{equation}
		F^\p(t)\geq \int\limits_{\R^d} \varrho\abs{\U}^2(t,\cdot)\,dx\mbox{ for }t\in(0,T^*).\label{eq-3}
	\end{equation}
	By using finite speed of propagation, Lemma \ref{lemma:fsp}, we observe that
	\begin{equation}
		F(t)=\int\limits_{\R^d}x\cdot \varrho\U\,dx=\int\limits_{B(t)}x\cdot \varrho\U\,dx.
	\end{equation}
	By using H\"older's inequality, we get
	\begin{equation}
		F^2(t)\leq \left(\int\limits_{B(t)}\abs{x}^2\varrho\,dx\right)\left(\int\limits_{B(t)}\varrho\abs{\U}^2\,dx\right). \label{bound-1}
	\end{equation}
	We note that
	\begin{align}
		\int\limits_{B(t)}\abs{x}^2\varrho\,dx
		\leq (R+\si t)^2\int\limits_{B(t)}\varrho(t,\cdot)\,dx
		&=(R+\si t)^2\left[\int\limits_{B(t)}\bar{\varrho}\,dx+M(0)\right]\nonumber\\
		&=(R+\si t)^2\left[\int\limits_{B_0}\varrho_0\,dx-\int\limits_{B_0}\bar{\varrho}\,dx+\int\limits_{B(t)}\bar{\varrho}\,dx\right]\nonumber\\
		&\leq \omega_d(R+\si t)^{d+2}\norm{\varrho_0}_{L^\f(\R^d)}.\label{bound-2}
	\end{align}
	Combining \eqref{eq-3}, \eqref{bound-1} and \eqref{bound-2}, we get
	\begin{equation*}
		F^2(t)\leq \omega_d(R+\si t)^{d+2}F^\p(t).
	\end{equation*}
	Therefore, by Fundamental Theorem of Calculus, we have
	\begin{equation*}
		-\frac{1}{d+1}\left[\frac{1}{(R+\si t)^{d+1}}-\frac{1}{R^{d+1}}\right]\leq \omega_d\norm{\varrho_0}_{L^\f(\R^d)}\left[\frac{1}{F(0)}-\frac{1}{F(t)}\right],
	\end{equation*}
	or equivalently,
	\begin{equation*}
		\frac{\omega_d\norm{\varrho_0}_{L^\f(\R^d)}}{F(0)}\geq 	\frac{\omega_d\norm{\varrho_0}_{L^\f(\R^d)}}{F(t)}+\frac{1}{(d+1)R^{d+1}}-\frac{1}{d+1}\frac{1}{(R+\si t)^{d+1}}.
	\end{equation*}
	If $t\rr\f$, we have
	\begin{equation*}
		\frac{\omega_d\norm{\varrho_0}_{L^\f(\R^d)}}{F(0)}\geq 	\frac{1}{(d+1)R^{d+1}},
	\end{equation*}
	which contradicts with \eqref{condition-Sideris}.
\end{proof}

\section{Global existence of smooth solution}
	In this section, we are going to prove the global existence of smooth solution to \eqref{eqn-isen-1}--\eqref{eqn-isen-2} when initial data satisfying the condition \descref{hyp-1}{(G-1)}--\descref{hyp-3}{(G-3)}. Our proof is relying on the technique of \cite{Grassin}. We first consider the following vectorial Burgers equation,
\begin{equation}\label{eqn-Burgers-multi}
\left.\begin{array}{rl}
	\pa_t\V+\V\cdot\nabla_x\V&=0,\mbox{ for }x\in\R^d,t>0,\\
	\V(0,x)&=\U_0\mbox{ for }x\in\R^d.
\end{array}	\right\}
\end{equation}
By method of characteristics, we can solve the equation \eqref{eqn-Burgers-multi} for small enough time $t_0$ and till the time of existence we have 
\begin{equation}
	\V(t,X(t,x_0))=\U_0(x_0)\mbox{ where }X(t,x_0)=x_0+t\U_0(x_0).
\end{equation}
We can calculate 
\begin{equation*}
	\nabla_x\V(t,X(t,x_0))=(\mathbb{I}_d+t\nabla_x\U_0(x_0))^{-1}\nabla_x\U_0(x_0)\mbox{ for }t\in[0,t_0).
\end{equation*}
Due to the assumption \descref{hyp-2}{(G-2)}, we can see that $	\norm{\nabla_x\V(t,X(t,\cdot))}_{L^\f(\R)}\leq C_1\norm{\nabla_x\U_0}_{L^\f(\R^d)}$ for some $C_1>0$ independent of $t$. Hence, the solution $\V(t,\cdot)$ exists for all time $t>0$. We further have the following estimate from \cite{Grassin},
\begin{proposition}[Grassin, \cite{Grassin}]
	Let $\al>0$, $m>1+\frac{d}{2}$ and $\U_0:\R^d\rr\R^d$ such that the following holds
	\begin{description}
		\descitem{(A-1)}{A-1} $\nabla_x^2\U_0\in H^{m-1}(\R^d)$ and $\nabla_x\U_0\in L^\f(\R^d)$.
		
		\descitem{(A-2)}{A-2} $\inf\left\{\nabla_x\U_0(x):\xi\otimes\xi;\,\xi\in\mathbb{S}^{d-1},x\in\R^d\right\}\geq0$ where $\mathbb{S}^{d-1}=\{x\in\R^d;\,\abs{x}=1\}$. 
	\end{description}
	We define
	\begin{equation}
		\Omega_\al:=\{x\in \R^d;\,\inf\limits_{\xi\in\mathbb{S}^{d-1}}\nabla_x\U_0(x):\xi\otimes\xi\geq\al\}.
	\end{equation}
	Let $\overline{\U}$ be the unique smooth solution to \eqref{eqn-Burgers-multi} with initial data $\U_0$. Then the following holds,
	\begin{enumerate}
		\item $\nabla_x\overline{\U}=\frac{1}{1+t}\mathbb{I}_d+\frac{1}{(1+t)^2}\mathcal{R}(t,x)$ where $\mathcal{R}(t,x)$ is bounded on every compact subset of $\Omega_\al(t)$ defined as 
		\begin{equation}
			\Omega_\al(t):=\left\{x;\,x=x_0+t\U_0(x_0),x_0\in\Omega_\al\right\}.
		\end{equation}
		\item $\norm{\nabla_x^2\overline{\U}(t,\cdot)}_{L^2(\Omega_\al(t))}\leq c_j(1+t)^{\frac{d}{2}-j-1}$ for $j=2,\dots,m+1$.
		\item $\norm{\nabla_x^2\overline{\U}(t,\cdot)}_{L^\f(\Omega_\al(t))}\leq c_0(1+t)^{-3}$.
	\end{enumerate}
\end{proposition}
As in \cite{Grassin}, following \cite{KMU} we consider the quantity $\pi$ defined as $\pi=\sqrt{\frac{\gamma-1}{4\gamma}}\varrho^\frac{\gamma}{2}$ and the symmetrization of \eqref{eqn-isen-1}--\eqref{eqn-isen-2} as below,
\begin{align}
	(\pa_t+\U\cdot\nabla_x)\pi+C_1\pi\dv_x\U&=0,\label{sym-E-1}\\
	(\pa_t+\U\cdot\nabla_x)\U+C_1\pi\nabla_x\pi&=0,\label{sym-E-2}
\end{align}
with $C_1=\frac{\gamma-1}{2}$. The system \eqref{sym-E-1}--\eqref{sym-E-2} can be written as quasi-linear form,
\begin{equation}\label{eqn-quasi-lin}
	\pa_t U+\sum\limits_{j=1}^{d}A_j(U)\pa_jU=0\mbox{ where }U=(\varrho,\U)^T,
\end{equation}
and $A_j(U)\in\mathbb{M}_{d+1}(\R)$ is defined as 
\begin{equation}
	A_j(U)=u_j\mathbb{I}_{d+1}+C_1\pi e_j\otimes e_j\mbox{ for }j=1,\dots,d
\end{equation}
and $\{e_0,e_1,\cdots,e_d\}$ is the standard basis of $\R^{d+1}$. Note that $\U_0$ is not integrable. Therefore, we approximate $\U_0$ by $\U_0^\la$ as $\U^\la_0=\psi_\la\U_0$ where $\psi\in C_0^\f(\R^d)$ such that
\begin{equation*}
	\psi_\la(x)=1\mbox{ for }\abs{x}\leq \la.
\end{equation*}
We choose $\la>0$ large enough such that $supp(\varrho_0)\subset \{x\in\R^d;\abs{x}\leq \la\}$.  For system \eqref{eqn-quasi-lin} with initial data $(\pi_0,\U_0^\la)$, smooth solution exists  in $[0,t_0)$ for small enough $t_0>0$ (see \cite{Chemin}). Let us consider $\mathcal{Q}$ defined as  
\begin{equation*}
	\mathcal{Q}_\al=\{(t,x)\in[0,\f)\times\R^d;x=y+t\nabla_x\U_0(y),y\in \Omega_\al,\abs{y}\leq\la\}.
\end{equation*}
Similar to \cite{Grassin}, we conclude that $(\pi(t,x),\U(t,x))=(0,\V(t,x))$ for $(t,x)\in [0,\f)\times\left(\R^d\setminus \mathcal{Q}_\al\right)$ by using the following  uniqueness result.
\begin{proposition}[Grassin ,\cite{Grassin}]
	Let $T_0>0$ and $W=(\pi,\U)\in C^1([0,T_0];H^{m-1}(\R^d))\cap C([0,T_0]H^{m}(\R^d))$ be a classical solution to \eqref{sym-E-1}--\eqref{sym-E-2} corresponding to initial data $W_0\in H^m(\R^d)$ for some $m>1+\frac{d}{2}$. Let $\overline{W}\in C^1([0,T_0]\times\R^d)$ also be a solution to \eqref{sym-E-1}--\eqref{sym-E-2} corresponding to initial data $\overline{W}_0$ such that $W_0(x)=\overline{W}_0(x)$ for $x\in B(x_0,r)$ for some $x_0\in\R^d,r>0$. Consider $M>0$ and $Q_T>0$ defined as 
	\begin{align*}
		M&=\sup\left\{(C_1\abs{\pi}+\abs{\U})(t,x);\,(t,x)\in D\right\}\mbox{ where }D=B(x_0,r)\times[0,T_0],\\
		Q_T&=\left\{(t,x);\,t\in[0,T],x\in B(x_0,r-Mt)\right\}\mbox{ for }T\in[0,T_1]\mbox{ with }T_1=\min\left\{T_0,r/M\right\}.
	\end{align*}
	Then we have
	\begin{equation}
		W(t,x)=\overline{W}(t,x)\mbox{ for }(t,x)\in Q_T.
	\end{equation}
\end{proposition}
Now, from system \eqref{eqn-Burgers-multi} and \eqref{sym-E-1}--\eqref{sym-E-2} we obtain
\begin{equation}\label{sym-diff-u-v}
	\left.\begin{array}{rl}
		(\pa_t+\W\cdot\nabla_x)\pi+C_1\pi\dv_x\W&=-\nabla_x\V\cdot\pi-C_1\pi\dv_x\V,\\
		(\pa_t+\W\cdot\nabla_x)\W+C_1\pi\nabla_x\pi&=-(\V\cdot\nabla_x)\W-(\W\nabla_x)\V,
	\end{array}	\right\}
\end{equation}
where $\W=\U-\V$. Let us denote $V=(0,\V)$ and $U=(\pi,\W)$. Then $U$ satisfies
\begin{equation}\label{eqn-U}
	\pa_tU+\sum\limits_{j=1}^dF_j(U)\pa_jU=-G(\nabla_xV,U)-\sum\limits_{j=1}^dv_j\pa_jU,
\end{equation}
where $G(\nabla_xV,U)=\begin{pmatrix}
	C_1\pi\dv_x\V\\
	\W\cdot\nabla_x\V
\end{pmatrix}$. We define
\begin{equation}
	\Gamma_k(t):=\left[\int\limits_{\R^d}\nabla_x^kU(t,x)\cdot\nabla_x^kU(t,x)\,dx\right]^\frac{1}{2}.
\end{equation}
In order to get the decay estimate for $\W$ we consider,
\begin{equation}
	\Gamma(t)=\sum\limits_{k=0}^{m}(1+t)^{\de_k}\Gamma_k(t)\mbox{ where }\de_k=k+b-a
\end{equation}
 and $b$ is defined as below 
\begin{equation}
	b=\left\{\begin{array}{ll}
		1-\frac{d}{2}&\mbox{ if }\de\geq\de_c=1+\frac{2}{d},\\
		\frac{\de-1}{2}d-\frac{d}{2} &\mbox{ if }1<\de<\de_c.
	\end{array}\right.
\end{equation}
We will choose $a$ later. By using \eqref{eqn-U} we can obtain
\begin{equation*}
	\frac{1}{2}\frac{d}{dt}\int\limits_{\R^d}\nabla_x^kU(t,x)\cdot\nabla_x^kU(t,x)\,dx=\int\limits_{\R^d}\mathcal{K}_k(U)(t,x)\,dx+\int\limits_{\R^d}\mathcal{A}_k(U,V)(t,x)\,dx,
\end{equation*}
where 
\begin{align*}
	\mathcal{K}_k(U)&=-\nabla_x^kU\cdot\left[\nabla_x^k\left(\sum\limits_{j=1}^{d}A_j(U)\pa_jU\right)-\sum\limits_{j=1}^{d}A_j(U)\pa_j\nabla_x^kU\right]+\frac{1}{2}\sum\limits_{j=1}^{d}\nabla_x^kU\cdot\pa_jA_j(U)\nabla_x^kU,\\
	\mathcal{A}_k(U,V)&=-\nabla_x^kU\cdot\nabla_x^k\left(B(\nabla_xV,U)\right)-\nabla_x^kU\cdot\left[\nabla_x^k\left(\V\cdot\nabla_xU\right)-\V\cdot\nabla_x\nabla_x^kU\right]\\
	&+\frac{1}{2}\dv_x(\V)\nabla_x^kU\cdot\nabla_x^kU.
\end{align*}
By using the similar argument as in \cite[Proposition 3]{Grassin}, we show that
\begin{proposition}\label{prop-2}
	Let $m>1+\frac{d}{2}$. Then there exists $C>0$ depending on $m,d$ and $\bar{C}>0$ depending on $m,d,\al,\norm{\nabla_x\U_0}_{L^\f(\R^d)},\norm{\nabla_x^2\U_0}_{H^{m-1}(\R^d)}$ such that the following holds,
	\begin{align*}
		\abs{\int\limits_{\R^d}\mathcal{K}_k(t,x)\,dx}&\leq C\norm{\nabla_xU}_{L^\f(\R^d)}\Gamma_k^2(t),\\
		\frac{k+s}{1+t}\Gamma_k^2(t)+\int\limits_{\R^d}\mathcal{A}_k(U,V)(t,x)\,dx&\leq \bar{C}\Gamma_k(t)\Gamma(t)(1+t)^{-\de_k-2}.
	\end{align*}
\end{proposition}
Now, with the help of above estimates, we are going to prove Proposition \ref{theorem:global-existence}. The argument is similar to the proof of \cite[Theorem 1]{Grassin}.
\begin{proof}[Proof of Proposition \ref{theorem:global-existence}:] By using Proposition \ref{prop-2} we have
	\begin{equation}
		\frac{1}{2}\frac{d}{dt}\Gamma_k^2(t)+\frac{k+s}{1+t}\Gamma_k^2\leq C\norm{\nabla_xU}_{L^\f(\R^d)}\Gamma_k^2(t)+\bar{C}\Gamma_k(t)\Gamma(t)(1+t)^{-\de_k-2}.
	\end{equation}
Setting $a=1+s+\frac{d}{2}>1$, we get
\begin{equation}
	\frac{d\Gamma(t)}{dt}+\frac{a}{1+t}\Gamma(t)\leq C(\Gamma(t))^2+\frac{\bar{C}}{(1+t)^2}\Gamma(t).
\end{equation}
By \cite[Proposition 4]{Grassin}, there exists $\e_0>0$ such that $\Gamma(t)\leq C_3(1+t)^{-a}$ for some $C_3>0$ provided $\Gamma(0)<\e_0$. This completes the proof of Proposition \ref{theorem:global-existence}.
\end{proof}

\bigskip
\noindent\textbf{Acknowledgement.} SSG would like to express thanks to the Department of Atomic Energy, Government of India, under project no. 12-R\&D-TFR-5.01-0520 for support. AJ acknowledges the support of Fondation Sciences Math\'ematiques de Paris-FSMP. 

\bigskip
\noindent\textbf{Disclosure statement:} On behalf of all authors, the corresponding author states that there is no conflict of interest.

\end{document}